\documentclass{m2an}
\usepackage{graphicx}
\begin{document}
\title{Magnus-type integrator for the finite element discretization of semilinear parabolic non-autonomous SPDEs driven by additive noise}
%
\author{Jean Daniel Mukam}\address{Fakult\"{a}t f\"{u}r Mathematik, Technische Universit\"{a}t Chemnitz, 09126 Chemnitz, Germany, email: \texttt{jean.d.mukam@aims-senegal.org}}
\author{Antoine Tambue}\address{Department of Computing Mathematics and Physics,  Western Norway University of Applied Sciences, Inndalsveien 28, 5063 Bergen.
Center for Research in Computational and Applied Mechanics (CERECAM), and Department of Mathematics and Applied Mathematics, University of Cape Town,
7701 Rondebosch, South Africa. The African Institute for Mathematical Sciences(AIMS) of South Africa, email: \texttt{antonio@aims.ac.za}}
%
\date{The dates will be set by the publisher}
\begin{abstract} 
In this  paper, we investigate a numerical approximation of a general second order 
 semilinear parabolic  non-autonomous   stochastic partial differential equation (SPDE)
 driven by  additive noise.  Numerical approximations for autonomous SPDEs are thoroughly 
 investigated in the literature while the   non-autonomous case is not yet well understood.
   We discretize the non-autonomous SPDE in space  by the finite element method  and in time by the Magnus-type integrator.
  We  provide a strong convergence proof of the fully discrete scheme toward the mild solution 
  in the root-mean-square $L^2$  norm.  Appropriate assumptions on the drift term  and the noise allow to achieve optimal 
  convergence order in time  greater than $1/2$, without any logarithmic reduction of convergence order in time.
  In particular, for trace class noise, we achieve optimal convergence orders  $\mathcal{O}\left(h^{2-\epsilon}+\Delta t\right)$, where $\epsilon$ is a positive number small enough.
  Numerical simulations are provided to illustrate our theoretical results.
 \end{abstract}
\begin{resume} 
Dans ce papier, nous investigons l'approximation num\'{e}rique d`equations aux deriv\'{e}es partielles (EDP) stochastique  s\'{e}milin\'{e}aire et 
non autonome avec un bruit additif. L'approximation num\'{e}rique d'EPD stochastique  autonome est largement \'{e}tudi\'{e}e dans la litt\'{e}rature scientifique, 
tandis que le cas non autonome reste encore tr\'{e}s peu connu. Le but de ce papier est d`investiguer le cas non autonome avec un bruit additif.
L'EDP stochastique est discretis\'{e}e en espace par la m\'{e}thode des \'{e}lements finis et en temps par un schema exponentiel de type Magnus. 
En plus, sous des hypoth\`{e}ses appropri\'{e}s, nous obtenons un ordre de convergence en temps sup\'{e}rieur a $1/2$, sans aucune reduction logarithmique. 
En particulier, pour un bruit de trace fini, nous  obtenons une convergence de la forme $\mathcal{O}\left(h^{2-\epsilon}+\Delta t\right)$, o\`{u} $\epsilon$ 
est un nombre r\'{e}el positif et suffisament petit. Les simulations num\'{e}riques pour illustrer les r\'{e}sultats th\'{e}oriques sont aussi faites. 
 \end{resume}
\maketitle
\section*{Introduction}
We consider  the numerical approximations of  the following semilinear parabolic  non-autonomous  SPDE driven by additive noise 
 \begin{eqnarray}
 \label{model1}
  dX(t)=[A(t)X(t)+F\left(t,X(t)\right)]dt+dW(t),\quad
  X(0)=X_0,\quad t\in (0,T],
 \end{eqnarray}
 in the Hilbert space $L^2(\Lambda)$, where $\Lambda$ is a bounded domain of $\mathbb{R}^d$, $d=1,2,3$ and $T>0$. 
 The family of the unbounded linear operators $A(t)$  are not necessarily self-adjoint. Each $A(t)$  is assumed to generate 
 an analytic semigroup $S_t(s):= e^{A(t)s}$.  Precise assumptions on $A(t)$ and $F$  to ensure the existence of the unique
 mild solution of \eqref{model1} are given in the next section. 
 The random initial data is denoted by $X_0$. We denote by $(\Omega, \mathcal{F}, \mathbb{P})$ a probability 
 space with a filtration $(\mathcal{F}_t)_{t\in[0,T]}\subset \mathcal{F}$ that fulfills the usual conditions, see e.g.,  \cite[Definition 2.1.11]{Prevot}.
 The noise term $W(t)$ is assumed to be a $Q$-Wiener process defined on a filtered probability space $(\Omega, \mathcal{F}, \mathbb{P}, \{\mathcal{F}_t\}_{t\in[0,T]})$,
 where the covariance operator $Q : H\longrightarrow H$ is assumed to be linear, self adjoint and positive definite.
It is well known (see e.g., \cite{Prevot}) that the noise can be represented as 
\begin{eqnarray}
\label{covariance}
W(t,x)=\sum_{i=0}^{\infty}\sqrt{q}_ie_i(x)\beta_i(t),
\end{eqnarray} 
where $(q_i,e_i)_{i\in\mathbb{N}}$ are the eigenvalues and eigenfunctions of the covariance operator $Q$, and $(\beta_i)_{i\in\mathbb{N}}$
are independent and identically distributed standard Brownian motion. 
Autonomous systems are not realistic to model  phenomena in many fields such 
as quantum fields theory, electromagnetism, nuclear physics, see e.g., \cite[Section 7]{Blanes} and references therein.
Numerical solutions of  \eqref{model1}  based on  implicit, explicit Euler 
methods and exponential integrators with $A(t)=A$, where $A$ is  selfadjoint  are thoroughly investigated in the literature, see e.g., 
\cite{Jentzenal,Kruse,WangR,Lord2}.  If we turn our attention to the case of $A(t)=A$,
with $A$ not necessary self adjoint, the list of references become remarkably short, see e.g.,  \cite{Lord1,Tambue1}. 
The numerical approximation in time of the  deterministic counter part of \eqref{model1} with time dependent coefficient
$A(t)$ was investigated in \cite{Gonzalezal,Hippal,Hochbruck,Tambue2}, where Magnus-type integrator \cite{Magnus}
was used in \cite{Gonzalezal,Hochbruck,Tambue2};  and a new exponential integrator was used in \cite{Hippal}. 
Numerical approximation for  non-autonomous SPDE \eqref{model1} is not yet well understood  due to the complexity of the linear operator $A(t)$ 
and its semigroup $S_t(s)= e^{A(t)s}$.
Recently numerical scheme for stochastic model \eqref{model1} driven by multiplicative noise with time dependent 
linear operator $A(t)$  was investigated in \cite{Tambue3}, where the time discretization was done using the Magnus-type integrator. 
The optimal convergence  order in time in \cite{Tambue3} was $1/2$. 
This is the optimal convergence order when dealing with multiplicative noise with schemes based on Euler approximations 
(namely explicit Euler method, linear implicit Euler method, exponential Euler, exponential Rosenbrock-Euler).
In fact even for stochastic ordinary differential equation (SODE) driven by multiplicative noise, 
the Euler type method achieves optimal order $1/2$, see e.g., \cite{Clark}, whereas when dealing with SODE driven by additive noise 
the optimal convergence order is $1$, see e.g. \cite{Kloeden}.
In this paper, we extend that result to the SPDE \eqref{model1} and prove that the Magnus-type integrator applied 
to SPDE \eqref{model1} achieves an optimal order $1$ in time.
The price to pay is that we require additional assumptions on the nonlinear function $F$ than the only standard 
Lipschitz condition. An important ingredient to achieve that optimal convergence order is the application of Taylor's
formula in Banach space  to the drift function, see  \secref{TalyorBanach}. It is worth to mention that
such approach and assumptions on the nonlinear drift function  $F$ were also used in \cite{Kloedenal,WangR,Tambue1,Lord2}
for exponential integrators and semi-implicit Euler method for autonomous SPDEs driven by additive noise to achieve optimal 
convergence order $1$ in time.  
Due to the complexity   of  the  linear operator and the corresponding  semi discrete linear operator after  space discretisation,
novel additional technical estimates are provided on the terms involving the noise to achieve higher convergence order, 
see e.g.,  \lemref{fonda} and  \secref{Noiseestimate}.  The result indicates how the convergence orders depend on the 
regularity of the initial data and the noise.  More precisely, the fully discrete scheme achieves convergence order 
$\mathcal{O}\left(h^{\beta}+\Delta t^{\beta/2}\right)$, where $\beta$ is  defined in  \assref{assumption1}.
We emphasize that comparing with results for autonomous SPDES with not necessary self adjoint, here we achieve optimal convergence order
$1$ in time for the border case $\beta=2$, instead of  sub-optimal convergence order $1-\epsilon$ obtained in \cite{Tambue1,Jentzenal}. 
The optimal convergence orders  achieved in \cite{Wang1,Kruse,WangR},  where due to  sharp integral estimates and optimal regularity estimates in \cite{KruseLarsson}.
Note that key ingredient to achieve optimal regularity estimates in \cite{KruseLarsson} is the spectral decomposition of the linear operator $A$. This cannot directly 
applied to the case of time dependent and not necessarily self-adjoint operator $A(t)$  due to its complexity and its associated semigroup $S_s(t)=e^{A(s)t}$.  
In  this paper, \lemsref{sharp} and \ref{sharpsecond} provide appropriate ingredients to fill the gap. 

The rest of this paper is organised as follows.  \secref{wellposed} provides the general setting, the numerical scheme and the main result.
 In  \secref{convergenceproof}, we provide some preparatory results and  present the proof of the main results.  \secref{experiment} 
provides some numerical experiments to sustain our theoretical results.
\section{Mathematical setting, numerical scheme and main results}
\label{wellposed}
\subsection{Notations and main assumptions}
\label{notation}
Let  $(H,\langle.,.\rangle_H,\Vert .\Vert)$ be an separable Hilbert space.  For all $p\geq 2$ and for a Banach space $U$,
we denote by $L^p(\Omega, U)$ the Banach space of all equivalence classes of $p$ integrable $U$-valued random variables. Let $L(U,H)$ be 
 the space of bounded linear mappings from $U$ to $H$ endowed with the usual  operator norm $\Vert .\Vert_{L(U,H)}$. By  $\mathcal{L}_2(U,H):=HS(U,H)$,
 we  denote the space of Hilbert-Schmidt operators from $U$ to $H$ equipped with the norm 
 $\Vert l\Vert^2_{\mathcal{L}_2(U,H)}:=\sum\limits_{i=1}^{\infty}\Vert l\psi_i\Vert^2$,  $l\in \mathcal{L}_2(U,H)$, where $(\psi_i)_{i=1}^{\infty}$ 
 is an orthonormal basis of $U$. Note that this definition is independent of the orthonormal basis of $U$. 
 For simplicity, we use the notations $L(U,U)=:L(U)$ and $\mathcal{L}_2(U,U)=:\mathcal{L}_2(U)$. 
 For all $l\in L(U,H)$ and $l_1\in\mathcal{L}_2(U)$ we have $ll_1\in\mathcal{L}_2(U,H)$ and 
\begin{eqnarray}
\label{chow1}
\Vert ll_1\Vert_{\mathcal{L}_2(U,H)}\leq \Vert l\Vert_{L(U,H)}\Vert l_1\Vert_{\mathcal{L}_2(U)},
\end{eqnarray}
see e.g., \cite{Chow}. 
   The covariance operator  $Q : H\longrightarrow H$ is assumed to be positive and self-adjoint. Throughout this paper $W(t)$ is a $Q$-wiener process. 
 The space of Hilbert-Schmidt operators from  $Q^{1/2}(H)$ to $H$ is denoted by $L^0_2:=\mathcal{L}_2(Q^{1/2}(H),H)=HS(Q^{1/2}(H),H)$. As usual, $L^0_2$ is equipped with the norm 
 \begin{eqnarray}
 \Vert l\Vert_{L^0_2} :=\Vert lQ^{1/2}\Vert_{HS}=\left(\sum\limits_{i=1}^{\infty}\Vert lQ^{1/2}e_i\Vert^2\right)^{1/2}, \quad  l\in L^0_2,
 \end{eqnarray}
where $(e_i)_{i=1}^{\infty}$ is an orthonormal basis  of $H$.
This definition is independent of the orthonormal basis of $H$. 
For an $L^0_2$- predictable stochastic process $\phi :[0,T]\times \Lambda\longrightarrow L^0_2$ such that
\begin{eqnarray}
\int_0^t\mathbb{E}\left\Vert \phi (s)Q^{1/2}\right\Vert^2_{\mathcal{L}_2(H)}ds<\infty,\quad t\in[0,T],
\end{eqnarray}
the following relation called It\^{o}'s isometry property holds
{\small
\begin{eqnarray}
\label{ito}
\mathbb{E}\left\Vert\int_0^t\phi(s) dW(s)\right\Vert^2=\int_0^t\mathbb{E}\left\Vert \phi(s)\right\Vert^2_{L^0_2}ds=\int_0^t\mathbb{E}\left\Vert\phi(s) Q^{1/2}\right\Vert^2_{\mathcal{L}_2(H)}ds,\quad t\in[0,T],
\end{eqnarray}
}
see e.g., \cite[Step 2 in Section 2.3.2]{Prato}  or \cite[Proposition 2.3.5]{Prevot}.

In the rest of this paper, we consider $H=L^2(\Lambda)$. 
To guarantee the existence of a unique mild solution of \eqref{model1} and for the purpose of the  convergence analysis, we make the following  assumptions.
\begin{assumption}
 \label{assumption1}
The initial data $X_0 : \Omega\longrightarrow H$ is assumed to be measurable and $X_0\in L^4\left(\Omega , \mathcal{D}\left((-A(0))^{\beta/2}\right)\right)$, $0\leq \beta\leq 2$.
 \end{assumption}
 We equip $V_{\alpha}(t) : = \mathcal{D}\left((-A(t))^{\alpha/2}\right)$, $\alpha\in \mathbb{R}$ with the norm $\Vert u\Vert_{\alpha,t} := \Vert (-A(t))^{\alpha/2}u\Vert$.
Due to \eqref{domaine}, \eqref{equivnorme1} and for the seek of  ease notations, we simply write $V_{\alpha}$ and $\Vert .\Vert_{\alpha}$.
We follow \cite{Seidler,WangR,Tambue1,Wang1} and  assume  that the nonlinear operator $F$ satisfies the following Lipschitz condition. 
\begin{assumption}
\label{assumption3}
The nonlinear operator $F : [0,T]\times H\longrightarrow H$ is assumed to be $\beta/2$-H\"{o}lder continuous with respect to the first variable
and Lipschitz continuous with respect to the second variable, i.e. there exists a positive constant $K_3$ such that 
\begin{eqnarray}
\Vert F(s,0)\Vert \leq K_3, \quad \Vert F(t, u)-F(s,v)\Vert \leq K_3\left(\vert t-s\vert^{\beta/2}+\Vert u-v\Vert\right),  \quad s, t\in[0,T],\quad u, v\in H.
\end{eqnarray} 
We also assume the drift function to be twice  differentiable with bounded derivative, i.e.  there  exists a constant $K_1>0$ such that
\begin{eqnarray}
\Vert F'(t,v)\Vert_{L(H)}&\leq& K_1, \quad \forall\, v\in H,\quad t\in[0,T]\\
 \Vert F''(t,u)(v_1,v_2)\Vert_{-\eta}&\leq& K_1\Vert v_1\Vert.\Vert v_2\Vert, \quad u, v_1,v_2\in H,\quad\text{for some }\, \eta\in[1,2),\quad t\in[0,T],
\end{eqnarray}
where the Fr\'{e}chet  first and second  order derivatives are taken respect  to the second variable.
\end{assumption}
\begin{assumption}
\label{assumption4}
We assume  the covariance operator $Q : H\longrightarrow H$ to satisfy
\begin{eqnarray}
\left\Vert (-A(0))^{\frac{\beta-1}{2}}Q^{\frac{1}{2}}\right\Vert_{\mathcal{L}_2(H)}<\infty, 
\end{eqnarray}
where $\beta$ is defined in Assumption \ref{assumption1}.
\end{assumption}
As in \cite{Gonzalezal,Hippal,Gonzalez}, we make the following assumptions on the family of  linear operator $A(t)$.
\begin{assumption}
\label{assumption2}
\begin{itemize}
\item[(i)]
We assume that $\mathcal{D}\left(A(t)\right)=D$, $0\leq t\leq T$ and the family of linear operators 
$A(t) : D\subset H\longrightarrow H$ to be uniformly sectorial on $0\leq t\leq T$, i.e. there exist constants $c>0$ and $\theta\in\left(\frac{1}{2}\pi,\pi\right)$ such that
\begin{eqnarray}
\left\Vert \left(\lambda\mathbf{I}-A(t)\right)^{-1}\right\Vert_{L(L^2(\Lambda))}\leq \frac{c}{\vert \lambda\vert},\quad \lambda\in S_{\theta},
\end{eqnarray}
where $S_{\theta}:=\left\{\lambda\in\mathbb{C} : \lambda=\rho e^{i\phi}, \rho>0, 0\leq \vert \phi\vert\leq \theta\right\}$. As in \cite{Hippal},
by a standard scaling argument, we assume $-A(t)$ to be invertible with bounded inverse.
\item[(ii)]
 We require the following Lipschitz conditions respect to the time
\begin{eqnarray}
\label{conditionB}
\left\Vert \left(A(t)-A(s)\right)(-A(0))^{-1}\right\Vert_{L(H)}&\leq& K_1\vert t-s\vert,\quad s,t\in[0, T],\\
 \label{conditionBB}
\left\Vert  (-A(0))^{-1}\left(A(t)-A(s)\right)\right\Vert_{L(D,H)}&\leq& K_1\vert t-s\vert,\quad s,t\in[0, T].
\end{eqnarray}
\item[(iii)]
 As we are dealing with non smooth data, we follow  \cite{Seidler}  and  assume that 
\begin{eqnarray}
\label{domaine}
\mathcal{D}\left(\left(-A(t)\right)^{\alpha}\right)=\mathcal{D}\left(\left(-A(0)\right)^{\alpha}\right),\quad 0\leq t\leq T,\quad 0\leq \alpha\leq 1
\end{eqnarray}
and there exists a positive constant $K_2$ such that  the following estimate holds
\begin{eqnarray}
\label{equivnorme1}
K_2^{-1}\left\Vert \left(-A(0)\right)^{\alpha}u\right\Vert\leq \left\Vert (-A(t))^{\alpha}u\right\Vert\leq K_2\left\Vert (-A(0))^{\alpha}u\right\Vert,\quad t\in[0,T],\quad u\in \mathcal{D}((-A(0))^{\alpha}).
\end{eqnarray} 
\end{itemize}
\end{assumption}
\begin{remark}
\label{remark1}
As a consequence of Assumption \ref{assumption2}, for all $\alpha\geq 0$ and $\gamma\in[0,1]$, there exists a constant $C_1\geq 0$ such that 
the following estimate holds uniformly in $t\in[0,T]$
\begin{eqnarray}
\label{smooth}
\left\Vert (-A(t))^{\alpha}e^{sA(t)}\right\Vert_{L(H)}\leq C_1s^{-\alpha},s>0,\\
\label{smootha}
 \left\Vert(-A(t))^{-\gamma}\left(\mathbf{I}-e^{sA(t)}\right)\right\Vert_{L(H)}\leq C_1s^{\gamma}, \quad s\geq 0.
\end{eqnarray}
\end{remark}
\begin{proposition}
\label{proposition1}
Let $\Delta(T):=\{(t,s) : 0\leq s\leq t\leq T\}$.
Under \assref{assumption2} there exists a unique evolution system \cite[Definition 5.3, Chapter 5]{Pazy} $U : \Delta(T)\longrightarrow L(H)$ such that
\begin{itemize}
\item[(i)] There exists a positive constant $K_0$ such that
\begin{eqnarray}
\Vert U(t,s)\Vert_{L(H)}\leq K_0,\quad 0\leq s\leq t\leq T.
\end{eqnarray}
\item[(ii)] $U(.,s)\in C^1(]s,T] ; L(H))$, $0\leq s\leq T$,
\begin{eqnarray}
\frac{\partial U}{\partial t}(t,s)=-A(t)U(t,s) \quad \text{and} \quad
\Vert A(t)U(t,s)\Vert_{L(H)}\leq \frac{K_0}{t-s},\quad 0\leq s<t\leq T.
\end{eqnarray}
\item[(iii)] $U(t,.)x\in C^1([0,t[ ; H)$, $0<t\leq T$, $x\in\mathcal{D}(A(0))$ and 
\begin{eqnarray}
\frac{\partial U}{\partial s}(t,s)=-U(t,s)A(s)x \quad \text{and}\quad
 \Vert A(t)U(t,s)A(s)^{-1}\Vert_{L(H)}\leq K_0, \quad 0\leq s\leq t\leq T.
\end{eqnarray}
\end{itemize}
\end{proposition}
\begin{proof}
See \cite[Theorem 6.1, Chapter 5]{Pazy}.
\end{proof}
\begin{theorem}
\label{theorem1} Let  \asssref{assumption1}, \ref{assumption3} and \ref{assumption2} (i)-(ii)   
be fulfilled. Then the non-autonomous problem \eqref{model1} has a unique mild solution $X(t)$, which takes the following form
\begin{eqnarray}
\label{mild1}
X(t)=U(t,0)X_0+\int_0^tU(t,s)F(s,X(s))ds+\int_0^tU(t,s)dW(s),
\end{eqnarray} 
where $U(t,s)$ is the evolution system of  \propref{proposition1}.
Moreover, there exists a positive constant $K_4$ such that 
\begin{eqnarray}
\label{borne1}
\sup_{0\leq t\leq T}\Vert X(t)\Vert_{L^2\left(\Omega, \mathcal{D}((-A(0))^{\beta/2}))\right)}\leq K_4\left(1+\Vert X_0\Vert_{L^2\left(\Omega,\mathcal{D}((-A(0))^{\beta/2}\right)}\right).
\end{eqnarray}
\end{theorem} 
\begin{proof}
See \cite[Theorem 1.3]{Seidler}.
\end{proof}

\subsection{Fully discrete scheme and main result}
\label{fullyscheme}
In the rest of this paper, we consider the family of linear operators $A(t)$ to be of second order of the following form
\begin{eqnarray}
\label{family}
A(t)u=\sum_{i,j=1}^d\frac{\partial}{\partial x_i}\left(q_{ij}(x,t)\frac{\partial u}{\partial x_j}\right)-\sum_{j=1}^dq_j(x,t)\frac{\partial u}{\partial x_j}.
\end{eqnarray}
We require the coefficients $q_{i,j}$ and $q_j$ to be smooth functions on the variable $x\in\overline{\Lambda}$ and H\"{o}lder-continuous with respect to $t\in[0,T]$. 
We further assume that there exists a positive constant $c_1$ such that the following  ellipticity condition holds
\begin{eqnarray}
\label{ellip}
\sum_{i,j=1}^dq_{ij}(x,t)\xi_i\xi_j\geq c_1\vert \xi\vert^2, \quad (x,t)\in\overline{\Lambda}\times [0,T].
\end{eqnarray}
 Under the above assumptions on $q_{ij}$ and $q_j$, it is  well known  that  the family of linear operators defined in \eqref{family} 
 fulfills   \assref{assumption2} (i)-(ii)  with $D=H^2(\Lambda)\cap H^1_0(\Lambda)$, see     \cite[Section 7.6]{Pazy} or \cite[Section 5.2]{Tanabe}.
 The above assumptions on $q_{ij}$ and $q_j$ also imply that  \assref{assumption2} (iii) is fulfilled, see e.g., \cite[Example 6.1]{Seidler} or \cite{Aman,Seely}.

As in  \cite{Fujita,Lord1}, we introduce two spaces $\mathbb{H}$ and $V$, such that $\mathbb{H}\subset V$, that depend on the boundary conditions 
for the domain of the operator $-A(t)$ and the corresponding bilinear form. For example, for Dirichlet  boundary conditions we take 
\begin{eqnarray}
V=\mathbb{H}=H^1_0(\Lambda)=\{v\in H^1(\Lambda) : v=0\quad \text{on}\quad \partial \Lambda\}.
\end{eqnarray}
For Robin  boundary condition and  Neumann  boundary condition, which is a special case of Robin boundary condition ($\alpha_0=0$), we take $V=H^1(\Lambda)$ and
\begin{eqnarray}
\mathbb{H}=\{v\in H^2(\Lambda) : \partial v/\partial v_A +\alpha_0v=0,\quad \text{on}\quad \partial \Lambda\}, \quad \alpha_0\in\mathbb{R}.
\end{eqnarray}
Using  Green's formula and the boundary conditions, we obtain the corresponding bilinear form associated to $-A(t)$  
\begin{eqnarray}
a(t)(u,v)=\int_{\Lambda}\left(\sum_{i,j=1}^dq_{ij}(x,t)\dfrac{\partial u}{\partial x_i}\dfrac{\partial v}{\partial x_j}+\sum_{i=1}^dq_i(x,t)\dfrac{\partial u}{\partial x_i}v\right)dx, \quad u,v\in V,
\end{eqnarray}
for Dirichlet boundary conditions and  
\begin{eqnarray}
a(t)(u,v)=\int_{\Lambda}\left(\sum_{i,j=1}^dq_{ij}(x,t)\dfrac{\partial u}{\partial x_i}\dfrac{\partial v}{\partial x_j}+\sum_{i=1}^dq_i(x,t)\dfrac{\partial u}{\partial x_i}v\right)dx+\int_{\partial\Lambda}\alpha_0uvdx,\quad u,v\in V.
\end{eqnarray}
for Robin  and Neumann boundary conditions. 
Using  G\aa rding's inequality, it holds that there exist two constants $\lambda_0$ and $c_0$ such that
\begin{eqnarray}
a(t)(v,v)\geq \lambda_0\Vert v \Vert^2_{1}-c_0\Vert v\Vert^2, \quad  v\in V,\quad t\in[0,T].
\end{eqnarray}
By adding and subtracting $c_{0}u $ on the right hand side of \eqref{model1}, we obtain a new family of linear operators that we still denote by  $A(t)$.
Therefore the  new corresponding   bilinear form associated to $-A(t)$ still denoted by $a(t)$ satisfies the following coercivity property
\begin{eqnarray}
\label{ellip2}
a(t)(v,v)\geq \; \lambda_0\Vert v\Vert_{1}^{2},\;\;\;\;\; v \in V,\quad t\in[0,T].
\end{eqnarray}
Note that the expression of the nonlinear term $F$ has changed as we have included the term $-c_{0}u$
in the new nonlinear term that we still denote by $F$.

The coercivity property \eqref{ellip2} implies that $A(t)$ and $A_h(t)$ \footnote{ Defined in \eqref{discreteoper}} 
are sectorial on $L^2(\Lambda)$ (uniformly in $h$), see e.g., \cite{Larsson2}. 
Therefore $A_h(t)$ and   $A(t)$ generate analytic semigroups denoted respectively by $S_{h,t}(s):=e^{sA_h(t)}$ and   $S_t(s)=e^{s A(t)}$  on $L^{2}(\Lambda)$  such that \cite{Henry}
\begin{eqnarray}
S_t(s)= e^{s A(t)}=\dfrac{1}{2 \pi i}\int_{\mathcal{C}} e^{ s\lambda}(\lambda I - A(t))^{-1}d \lambda,\;\;\;\;\;\;\;
\;s>0,
\end{eqnarray}
where $\mathcal{C}$  denotes a path that surrounds the spectrum of $A(t)$.
The coercivity  property \eqref{ellip2} also implies that $-A(t)$ is a positive operator and its fractional powers are well defined and 
  for any $\alpha>0$, we have
\begin{equation}
\label{fractional}
 \left\{\begin{array}{rcl}
         (-A(t))^{-\alpha} & =& \frac{1}{\Gamma(\alpha)}\displaystyle\int_0^\infty  s^{\alpha-1}{\rm e}^{sA(t)}ds,\\
         (-A(t))^{\alpha} & = & ((-A(t))^{-\alpha})^{-1},
        \end{array}\right.
\end{equation}
where $\Gamma(\alpha)$ is the Gamma function \cite{Henry}.  
 The domain  of $(-A(t))^{\alpha/2}$  are  characterized in \cite{Fujita,Elliot,Larsson2}  for $1\leq \alpha\leq 2$  with equivalence of norms as follows
\begin{eqnarray}
\mathcal{D}((-A(t))^{\alpha/2})=H^1_0(\Lambda)\cap H^{\alpha}(\Lambda)\hspace{1cm} 
\text{(for Dirichlet boundary condition)}\nonumber\\
\mathcal{D}(-A(t))=\mathbb{H},\quad \mathcal{D}((-A(t))^{1/2})=H^1(\Lambda)\hspace{0.5cm} \text{(for Robin boundary condition)}\nonumber\\
\Vert v\Vert_{H^{\alpha}(\Lambda)}\equiv \Vert ((-A(t))^{\alpha/2}v\Vert:=\Vert v\Vert_{\alpha},\quad \forall v\in \mathcal{D}((-A(t))^{\alpha/2}).\nonumber
\end{eqnarray}
The characterization of $\mathcal{D}((-A(t))^{\alpha/2})$ for $0\leq \alpha<1$ can be found in  \cite[Theorem 2.1 \& Theorem 2.2]{Nambu}.

Now, we turn our attention to the discretization of the problem \eqref{model1}.
We start by splitting the domain $\Lambda$ in finite  triangles. Let $\mathcal{T}_h$ be the triangulation with maximal length $h$ satisfying the usual
regularity assumptions, and $V_h\subset V$ be the space of continuous functions that are piecewise linear over the triangulation $\mathcal{T}_h$.
We consider the projection $P_h$ from $H=L^2(\Lambda)$ to $V_h$ defined for every $u\in H$ by 
\begin{eqnarray}
\label{proj1}
\langle P_hu, \chi\rangle_H=\langle u,\chi\rangle_H, \quad \phi, \chi \in V_h.
\end{eqnarray}
 For all $t\in[0,T]$, the discrete operator $A_h(t) :V_h\longrightarrow V_h$ is defined by 
 \begin{eqnarray}
 \label{discreteoper}
 \langle A_h(t)\phi,\chi\rangle_H=\langle A(t)\phi,\chi\rangle_H=-a(t)(\phi,\chi), \quad \phi,\chi\in V_h.
 \end{eqnarray}
 The coercivity property \eqref{ellip2}  implies that there exist constants $C_2>0$ and $\theta\in(\frac{1}{2}\pi,\pi)$ such that 
 \begin{eqnarray}
 \label{sectorial1}
 \Vert (\lambda\mathbf{I}-A_h(t))^{-1}\Vert_{L(H)}\leq \frac{C_2}{\vert \lambda\vert},\quad \lambda \in S_{\theta}
 \end{eqnarray}
 holds uniformly for $h>0$ and $t\in[0,T]$. See e.g.,  \cite{Larsson2} (2.9) or \cite{Fujita,Henry}. 
 The coercivity property \eqref{ellip2} also implies that the smooth properties \eqref{smooth}  and \eqref{smootha} hold 
 for $A_h$ uniformly on $h>0$ and $t\in[0,T]$, i.e. for all $\alpha\geq 0$ and $\gamma\in[0,1]$, there exist a positive constant $C_3$ 
 such that the following estimates hold uniformly on $h>0$ and $t\in[0,T]$, see e.g. \cite{Fujita,Henry}
 \begin{eqnarray}
 \label{smooth2}
 \left\Vert(-A_h(t))^{\alpha}e^{sA_h(t)}\right\Vert_{L(H)}\leq C_3s^{-\alpha}, \quad s>0, \\
 \label{smooth1}
  \left\Vert (-A_h(t))^{-\gamma}\left(\mathbf{I}-e^{sA_h(t)}\right)\right\Vert_{L(H)}\leq C_3s^{\gamma}, \quad s\geq 0.
 \end{eqnarray}
 The semi-discrete version of  \eqref{model1} consists of finding $X^h(t)\in V_h$, $t\in[0,T]$ such that
 \begin{eqnarray}
 \label{semi1}
 dX^h(t)=\left[A_h(t)X^h(t)+P_hF\left(t,X^h(t)\right)\right]dt+P_hdW(t),\quad t\in(0,T],\quad X^h(0)=P_hX_0.
 \end{eqnarray}
 Throughout this paper we take $t_m=m\Delta t\in[0,T]$, where $T=M\Delta t$ for $m, M\in\mathbb{N}$, $m\leq M$.
Following \cite{Tambue3}, we have the following fully discrete scheme for \eqref{model1}, called stochastic Magnus-type integrator (SMTI) for SPDEs 
\begin{eqnarray}
\label{scheme3}
X^h_{m+1}=e^{\Delta tA_{h,m}}X^h_m+\Delta t\varphi_1(\Delta tA_{h,m})P_hF\left(t_m, X^h_m\right)+e^{\Delta tA_{h,m}}P_h\Delta W_m, \quad m\geq 0,\quad X^h_0=P_hX_0,
\end{eqnarray}
 where $
 \Delta W_m :=W_{(m+1)\Delta t}-W_{m\Delta t}$,   $ A_{h,m}:=A_h(t_m) $
 and the linear operator $\varphi_1(\Delta t A_{h,m})$ is given  by 
\begin{eqnarray}
\varphi_1(\Delta tA_{h,m}):=\dfrac{1}{\Delta t}\int_0^{\Delta t}e^{(\Delta t-s)A_{h,m}}ds.
\end{eqnarray}
  Note that the numerical scheme \eqref{scheme3} can be written in the following integral form, useful for the error analysis
 \begin{eqnarray}
 \label{scheme4}
 X^h_{m+1}=e^{\Delta tA_{h,m}}X^h_m+\int_{t_m}^{t_{m+1}}e^{(t_{m+1}-s)A_{h,m}}P_hF\left(t_m, X^h_m\right)ds+\int_{t_m}^{t_{m+1}}e^{\Delta tA_{h,m}}P_hdW(s).
 \end{eqnarray}
Note also that an equivalent formulation of  the numerical scheme \eqref{scheme3}, easy for simulation is given by
\begin{eqnarray}
\label{scheme5}
X^h_{m+1}=X^h_m+P_h\Delta W_m+\Delta t\varphi_1(\Delta tA_{h,m})\left[A_{h,m}\left(X^h_m+P_h\Delta W_m\right)+P_hF\left(t_m, X^h_m\right)\right].
\end{eqnarray}
The following assumption will be needed in our convergence estimate to achieve optimal convergence order in time without any logarithmic reduction.
\begin{assumption}
\label{assumption5}
Let $A(t)=A^s(t)+A^{ns}(t)$, where $A^s(t)$ and $A^{ns}(t)$ are respectively the self-adjoint and the non self-adjoint parts of $A(t)$.
We assume that  the family $(\lambda_n(t))_{n\in \mathbb{N}}$ 
of positive eigenvalues of $-A^s(t)$ corresponding to the eigenvectors $(e_n(t))_{n\in \mathbb{N}}$ are such that for $x\in H$  
\begin{eqnarray}
 \label{supcond}
 \underset{0\leq t\leq T}{\sup} \lambda_n(t) < C(n), \,\quad  \underset{0\leq t\leq T}{\sup} (e_n(t),x) < C_1(x,n).
 \end{eqnarray}
 where  $C (n)$ and $C_1=C_1(x, n)$  are two positive constants.
 \end{assumption}
\begin{remark}
 Typical examples which fulfilled Assumption \ref{assumption5} are linear operators $A(t)$ defined in \eqref{family}
 with  bounded coefficients such that  $q_{ii}(x,t) > 0$ and $q_{ij}(x,t) = 0,\,\,i\neq j$ with \eqref{supcond}.
 Note that  \assref{assumption5} coincides with the assumptions made in \cite{KruseLarsson,Kruse,Wang1} on the constant self-adjoint operator $A$,
 where the authors also achieved optimal convergence orders. Note that these optimal convergence orders were due to the sharp integral estimate \cite{KruseLarsson}.
 In the case of non-autonomous  and non necessarily self adjoint operator,  \lemsref{sharp} and \ref{sharpsecond} are keys ingredients to achieve optimal
 convergence orders with no reduction. 
 \end{remark}
  In the rest of this paper $C$ denotes a generic constant that may change from one place to another. The numerical method being built, 
  we can now state its strong convergence result toward the mild solution, which is the  main result of this work.
\begin{theorem}\textbf{[Main result]}
\label{mainresult1}
Let  \asssref{assumption1}-\ref{assumption2} and \ref{assumption5} be fulfilled. Then the following error estimate holds
\begin{itemize}
\item[(i)] If $0\leq \beta<2$ then 
\begin{eqnarray}
\left(\mathbb{E}\Vert X(t_m)-X^h_m\Vert^2\right)^{1/2}\leq C\left(h^{\beta}+\Delta t^{\beta/2}\right).
\end{eqnarray}
\item[(ii)] If $\beta=2$ then 
\begin{eqnarray}
\left(\mathbb{E}\Vert X(t_m)-X^h_m\Vert^2\right)^{1/2}\leq C\left[h^{2}\left(1+\max\left(\ln\left(\frac{t_m}{h^2}\right),0\right)\right)+\Delta t\right].
\end{eqnarray}
\end{itemize}
\end{theorem}
\begin{remark}
If we relax \assref{assumption5}, then we obtain the following convergence result.
\begin{itemize}
\item[(i)] If $0< \beta<2$, the following error estimate holds
\begin{eqnarray}
\label{main1}
\left(\mathbb{E}\Vert X(t_m)-X^h_m\Vert^2\right)^{1/2}\leq C\left(h^{\beta}+\Delta t^{\beta/2-\epsilon}\right),
\end{eqnarray}
where $\epsilon>0$ is a positive number small enough. 
\item[(ii)] If $\beta=2$, then the following error estimate holds
\begin{eqnarray}
\label{main2}
\left(\mathbb{E}\Vert X(t_m)-X^h_m\Vert^2\right)^{1/2}\leq C\left[h^{2}\left(1+\max\left(\ln\left(\frac{t_m}{h^2}\right),0\right)\right)+\Delta t^{1-\epsilon}\right].
\end{eqnarray}
\end{itemize}
\end{remark}
\begin{remark}
\label{remark3} Note that as in \cite{Lord1,Jentzenal,WangR}, we can use the following  approximation  
\begin{eqnarray}
\int_{t_{m-1}}^{t_m}e^{A_{h,m}(t_m-s)}P_hF\left(s,X^h(s)\right)ds\approx \int_{t_{m-1}}^{t_m}e^{A_{h,m}\Delta t}P_hF\left(t_m,X^h_m\right)ds=\Delta te^{A_{h,m}\Delta t}P_hF\left(t_m,X^h_m\right).
\end{eqnarray}
This yields the following numerical Magnus-type integrator scheme
\begin{eqnarray}
\label{schemeanto}
Y^h_{m+1}=e^{\Delta tA_{h,m}}\left[Y^h_m+\Delta tP_hF\left(t_m, Y^h_m\right)+P_h\Delta W_m\right],\quad Y^h_0=P_hX_0.
\end{eqnarray}
Note that the convergence result in Theorem \ref{mainresult1} also holds for the numerical scheme \eqref{schemeanto}. The proof is similar of that of  \thmref{mainresult1}.
\end{remark}
\section{Proof of the main result}
\label{convergenceproof}
The proof of  the main result needs some preparatory results.
\subsection{Preparatory results}
\label{prepa}
The following lemma will be useful in our convergence proof. Its proof can be found in \cite{Tambue2}.
\begin{lemma}
\label{lemma0}
For any $\gamma\in[0,1]$, the following equivalence of norms holds  uniformly in $h>0$ and $t\in[0,T]$.
\begin{eqnarray}
\label{equidiscrete1}
K^{-1}\Vert (-(A_h(0))^{-\gamma}v\Vert\leq \Vert ((-A_h(t))^{-\gamma}v\Vert\leq K\Vert ((-A_h(0))^{-\gamma}v\Vert,\quad v\in V_h,\\
\label{equidiscrete2}
K^{-1}\Vert (-A_h(0))^{\gamma}v\Vert\leq \Vert (-A_h(t))^{\gamma}v\Vert\leq K\Vert (-A_h(0))^{\gamma}v\Vert,\quad v\in V_h.
\end{eqnarray}
\end{lemma}

\begin{lemma}
\label{lemma1a}
Under    \asssref{assumption4} and \ref{assumption2} (iii), the following estimate holds
\begin{eqnarray}
\left\Vert (-A_h(t))^{\frac{\beta-1}{2}}P_hQ^{\frac{1}{2}}\right\Vert_{\mathcal{L}_2(H)}<C,\quad t\in[0,T],\quad h>0.
\end{eqnarray}
\end{lemma}
\begin{proof}
For $0\leq\beta\leq 1$, it follows from \cite[Proposition 4.1]{TambueNg} that
\begin{eqnarray}
\label{br1}
\left\Vert (-A_h(0))^{\frac{\beta-1}{2}}P_hQ^{\frac{1}{2}}\right\Vert_{\mathcal{L}_2(H)}<C.
\end{eqnarray}
Therefore using \eqref{br1} and Lemma \ref{lemma0}  it follows  that
\begin{eqnarray}
\left\Vert (-A_h(t))^{\frac{\beta-1}{2}}P_hQ^{\frac{1}{2}}\right\Vert_{\mathcal{L}_2(H)}<C,\quad t\in[0,T],\quad \beta\in[0,1].
\end{eqnarray}
Let us   recall  the following estimate \cite[Lemma 1]{Tambue1} 
\begin{eqnarray}
\label{permut1}
\Vert (-A_h(0))^{\alpha}P_hv\Vert\leq C\Vert (-A(0))^{\alpha}v\Vert,\quad 0\leq \alpha\leq 1/2,\quad v\in\mathcal{D}((-A(0))^{\alpha}).
\end{eqnarray}
For $1\leq \beta\leq 2$, applying  \eqref{permut1} with $\alpha=\frac{\beta-1}{2}$ and using Assumption \ref{assumption4} yields
\begin{eqnarray}
\label{br5}
\left\Vert (-A_h(0))^{\frac{\beta-1}{2}}P_hQ^{\frac{1}{2}}\right\Vert_{\mathcal{L}_2(H)}&=&\sum_{i=1}^{\infty}\left\Vert (-A_h(0))^{\frac{\beta-1}{2}}P_hQ^{\frac{1}{2}}e_i\right\Vert \nonumber\\
&\leq& C\sum_{i=1}^{\infty}\left\Vert (-A(0))^{\frac{\beta-1}{2}}Q^{\frac{1}{2}}e_i\right\Vert\nonumber\\
&=& C\left\Vert (-A(0))^{\frac{\beta-1}{2}}Q^{\frac{1}{2}}\right\Vert_{\mathcal{L}_2(H)}\leq C,
\end{eqnarray} 
Therefore, it follows from \eqref{br5} by using \eqref{equidiscrete2} that 
\begin{eqnarray}
\left\Vert (-A_h(t))^{\frac{\beta-1}{2}}P_hQ^{\frac{1}{2}}\right\Vert_{\mathcal{L}_2(H)}\leq C,\quad t\in[0,T].
\end{eqnarray}
\end{proof}

The proof of the following lemma can be found in \cite{Tambue2}.
\begin{lemma}
\label{lemma0a} 
Under  \assref{assumption2}, the following estimates hold
\begin{eqnarray}
\label{ref3}
\Vert (A_h(t)-A_h(s))(-A_h(r))^{-1}u^h\Vert&\leq& C\vert t-s\vert\Vert u^h\Vert,\quad r,s,t\in[0,T],\quad u^h\in V_h,\\
\label{ref2}
\Vert (-A_h(r))^{-1}\left(A_h(s)-A_h(t)\right)u^h\Vert&\leq& C\vert s-t\vert\Vert u^h\Vert,\quad r,s,t\in[0,T],\quad u^h\in V_h.
\end{eqnarray}
\end{lemma}
\begin{remark}
\label{evolutionremark}
From  \lemref{lemma0a} and using the fact that $\mathcal{D}(A_h(t))=\mathcal{D}(A_h(0))$, it follows from  \cite[Theorem 6.1, Chapter 5]{Pazy} 
that there exists a unique evolution system $U_h :\Delta(T)\longrightarrow L(H)$, satisfying  \cite[(6.3), Page 149]{Pazy}.
\end{remark}

\begin{lemma}
\label{evolutionlemma}
Let  \assref{assumption2} be fulfilled. 
\begin{itemize}
\item[(i)] The following estimate holds
\begin{eqnarray}
\label{reste2}
 \Vert U_h(t,s)\Vert_{L(H)}\leq C,\quad 0\leq s\leq t\leq T.
\end{eqnarray}
\item[(ii)] For any $0\leq\alpha\leq 1$, $0\leq\gamma\leq 1$ and $0\leq s\leq t\leq T$, the following estimates hold
\begin{eqnarray}
\label{ae1}
\Vert (-A_h(r))^{\alpha}U_h(t,s)\Vert_{L(H)}&\leq& C(t-s)^{-\alpha},\quad r\in[0,T],\\
\label{ae3}
\Vert U_h(t,s)(-A_h(r))^{\alpha}\Vert_{L(H)}&\leq& C(t-s)^{-\alpha},\quad r\in[0,T],\\
\label{ae2}
 \Vert (-A_h(r))^{\alpha}U_h(t,s)(-A_h(s))^{-\gamma}\Vert_{L(H)}&\leq& C(t-s)^{\gamma-\alpha}, \quad r\in[0,T].
\end{eqnarray}
\item[(iii)] For any $0\leq s\leq t\leq T$,  the following useful estimate holds 
\begin{eqnarray}
\label{hen1}
\Vert \left(U_h(t,s)-\mathbf{I}\right)(-A_h(s))^{-\gamma}\Vert_{L (H)}&\leq& C(t-s)^{\gamma}, \quad 0\leq \gamma\leq 1,\\
\label{hen2}
\Vert \left (-A_h(r))^{-\gamma}(U_h(t,s)-\mathbf{I}\right)\Vert_{L (H)}&\leq& C(t-s)^{\gamma}, \quad 0\leq \gamma\leq 1.
\end{eqnarray}
\end{itemize}
\end{lemma}

\begin{proof}
The proof can be found in \cite{Tambue2}.
\end{proof}

\begin{remark}
For relatively smooth coefficients ($q_j\in C^1(\Lambda)$), the formal adjoint of $A(t)$ denoted by $A^*(t)$ is given by \cite[Section 6.2.3]{Evans}
 \begin{eqnarray}
 \label{adjoint1}
 A^*(t)=\sum_{i,j=1}^d\frac{\partial}{\partial x_j}\left(q_{ij}(x,t)\frac{\partial}{\partial x_i}\right)+\sum_{j=1}^dq_j(x,t)\frac{\partial}{\partial x_j}+\left(\sum_{j=1}^d\frac{\partial q_j}{\partial x_j}(x,t)\right)\mathbf{I},\quad t\in[0, T].
 \end{eqnarray}
 Therefore the self-adjoint part of $A(t)$ is given by 
 \begin{eqnarray}
 A^s(t)=\sum_{i,j=1}^d\frac{\partial}{\partial x_j}\left(q_{ij}(x,t)\frac{\partial}{\partial x_i}\right),\quad t\in[0,T].
 \end{eqnarray}
 The bilinear operator associated to  $A^s(t)$ is given by
 \begin{eqnarray}
 a^s(t)(u,v)=\sum_{i, j=1}^d\int_{\Lambda}q_{ij}(x,t)\frac{\partial u}{\partial x_i}\frac{\partial v}{\partial x_j}dx,\quad u,v \in V,\quad t\in[0,T].
 \end{eqnarray}
 The discrete version $A_h^s(t)$  of $A^s(t)$ is therefore given by $A^s_h(t): V_h\longrightarrow V_h$ such that 
 \begin{eqnarray}
 \left\langle A_h^s(t)\phi, \chi\right\rangle_H=\left\langle A^s(t)\phi, \chi\right\rangle_H=-a^s(t)(\phi, \chi),\quad \phi, \chi\in V_h.
 \end{eqnarray}
 Hence $A^s_h(t)$  satisfies also Assumption \ref{assumption5} and
 \begin{eqnarray}
 A_h(t)=A^s_h(t)+A^{ns}_h(t),
 \end{eqnarray}
 where $A^{ns}_h(t)$ is the non self adjoint part of $A_h(t)$.
\end{remark}
The following sharp integral estimate will be useful in our convergence analysis to avoid suboptimal convergence order and is a key ingredient 
to achieve optimal convergence order in time. It is an analogue of \cite[Lemma 3.2 (iii)]{KruseLarsson}  for evolution system.
\begin{lemma}
\label{sharp}
Let  \asssref{assumption2} and \eqref{assumption5} be fulfilled and let $0\leq \rho\leq 1$. Then the following estimate holds
\begin{eqnarray}
\label{sharpa}
\int_{\tau_1}^{\tau_2}\left\Vert (-A_h(\tau))^{\rho/2}S^h_r(\tau_2-r)\right\Vert^2_{L(H)}dr&\leq& C(\tau_2-\tau_1)^{1-\rho},\quad 0\leq\tau_1\leq \tau_2\leq T,\quad \tau\in[0,T],\\
\label{sharpb}
\int_{\tau_1}^{\tau_2}\left\Vert (-A_h(\tau))^{\rho/2}U_h(\tau_2, r)\right\Vert^2_{L(H)}dr&\leq& C(\tau_2-\tau_1)^{1-\rho},\quad 0\leq \tau_1\leq \tau_2\leq T,\quad \tau\in[0,T],\\
\label{sharpc}
\int_{\tau_1}^{\tau_2}\left\Vert U_h(\tau_2, r)(-A_h(\tau))^{\rho/2}\right\Vert^2_{L(H)}dr&\leq& C(\tau_2-\tau_1)^{1-\rho},\quad 0\leq \tau_1\leq \tau_2\leq T,\quad \tau\in[0,T].
\end{eqnarray}
\end{lemma}
\begin{proof}
We start with the estimate of \eqref{sharpa}. 
Let us recall that $A_h(r)=A^s_h(r)+A^{ns}_h(r)$, where $A^s_h(r)$ and $A^{ns}_h(r)$ are respectively the self adjoint and the non self adjoint parts of $A_h(r)$.
As in \cite{TambueNg}, we use the Zassenhaus formula \cite{Scholz,Magnus} to decompose the semigroup $S^h_r(t)$ as follows
\begin{eqnarray}
\label{sharp6}
S^h_r(t)=e^{A_h(r)t}=e^{\left(A_h^s(r)+A^{ns}_h(r)\right)t}=e^{A_h^s(r)t}e^{A^{ns}_h(r)t}
\prod_{k=2}^{\infty}e^{C^h_k(r,t)}, 
\end{eqnarray}
where the $C^h_k(r,t)$ are called Zassenhaus exponents \cite{Scholz}. Let us set
\begin{eqnarray}
\label{sharp7}
T^h_r(t):=e^{A^{ns}_h(r)t}\prod_{k=2}^{\infty}e^{C^h_k(r,t)}.
\end{eqnarray}
Therefore
\begin{eqnarray}
\label{sharp8}
S^h_r(t)=S^h_s(r,t)T^h_r(t),
\end{eqnarray}
where $S^h_s(r,t):=e^{A_h^s(r)t}$ is the semigroup generated by $A^s_h(r)$. Using the Baker-Campbell-Hausdorff representation \cite{Scholz,Mielnik,Casas},
it is well known that for non-commuting quantities $x$ and $y$, we have 
\begin{eqnarray}
\label{sharp9}
e^xe^y=e^{x\oplus y},
\end{eqnarray}
where the exponent $x\oplus y$ is given by an infinite Baker-Campbell-Hausdorff series of multiple commutators with rational coefficients
(see e.g.,  \cite[(1.1)]{Mielnik}  or  \cite[(1)-(2)]{Scholz}) and converges to $\log(e^xe^y)$. Using \eqref{sharp9}, by recurrence,  there exits $Z=Z(t, A^s_h(r), A^{ns}_h(r))$ such that the operator $T^h_r(t)$ can be written as
\begin{eqnarray}
T^h_r(t)=e^Z.
\end{eqnarray}
Therefore, $T^h_r(t)$ is uniformly bounded. Note that $\mathcal{D}\left(-A_h(r))\right)=\mathcal{D}\left(-A_h^s(r)\right)$, with the equivalence of norms, see e.g., \cite{Fujita,Larsson1}. So by  \cite[(3.3)]{Lions} and using Assumption \ref{assumption2} and Lemma \ref{lemma0}, we have $\mathcal{D}\left((-A_h(r))^{\alpha})\right)=\mathcal{D}\left((-A_h^s(r))^{\alpha}\right)$, $0\leq \alpha\leq1$, with the equivalence of norms.  Therefore, using \eqref{sharp9} and the boundedness of $T^h_r(t)$ yields
\begin{eqnarray}
\label{sharp10}
\int_{\tau_1}^{\tau_2}\left\Vert (-A_h(\tau))^{\rho/2}S^h_r(\tau_2-r)\right\Vert^2_{L(H)}dr&\leq& C\int_{\tau_1}^{\tau_2}\left\Vert (-A_h(r))^{\rho/2}S^h_r(\tau_2-r)\right\Vert^2_{L(H)}dr\nonumber\\
&\leq& C\int_{\tau_1}^{\tau_2}\left\Vert (-A_h(r))^{\rho/2}S^h_s(r, \tau_2-r)T^h_r(\tau_2-r)\right\Vert^2_{L(H)}dr\nonumber\\
&=& C\int_{\tau_1}^{\tau_2}\left\Vert (-A_h(r))^{\rho/2}S^h_s(r, \tau_2-r)P_hT^h_r(\tau_2-r)\right\Vert^2_{L(H)}dr\nonumber\\
&\leq & C\int_{\tau_1}^{\tau_2}\left\Vert (-A_h(r))^{\rho/2}S^h_s(r, \tau_2-r)P_h\right\Vert^2_{L(H)}dr\nonumber\\
&\leq& C\int_{\tau_1}^{\tau_2}\left\Vert(-A_h^s(r))^{\rho/2}S^h_s(r, \tau_2-r)P_h\right\Vert^2_{L(H)}dr.
\end{eqnarray}
From  \assref{assumption5}, it follows that there exists an increasing sequence of real numbers $(\lambda_n^h(r))_{n= 1}^{N_h}$ 
and eigenfunctions $(e^h_n(r))_{n=1}^{N_h}$ in $H$ such that $A_h(r)e^h_n(r)=\lambda^h_n(r)e^h_n(r)$ and
\begin{eqnarray}
0<\lambda^h_1(r)\leq \lambda_2^h(r)\leq\cdots\leq \lambda^h_{N_h}(r), 
\end{eqnarray}
where $dim(V_h)=N_h$.
From  the coercivity \eqref{ellip2}, there exists $C>0$ such that
\begin{eqnarray}
\inf_{0\leq t\leq T}\lambda^h_n(t) < C.
\end{eqnarray}
However as $e^h_n(t) $ tends to $e_n(t)$ when $h \rightarrow 0$, from \eqref{supcond} we also have
\begin{eqnarray}
\sup_{0\leq t\leq T}\lambda^h_n(t)<C(n) ,\quad \text{and}\quad \sup_{0\leq t\leq T}(x, e^h_n(t))<C_1(x,n),\quad n=1,\cdots,N_h,\quad x\in H.
\end{eqnarray}
Like in the proof of  \cite[Lemma 3.2 (iii)]{KruseLarsson}, using the expansion of $P_hx$ (with $x\in H$),  in terms of the eigenbasis $(e^h_n(r))_{n= 1}^{N_h}$ of the operator $A_h(r)$ and careful estimates yields
\begin{eqnarray}
\label{sharp11}
&&\int_{\tau_1}^{\tau_2}\left\Vert\left(-A_h^s(r)\right)^{\rho/2}S^h_s(r, \tau_2-r)P_hx\right\Vert^2dr\nonumber\\
&=&\int_{\tau_1}^{\tau_2}\left\Vert\sum_{n=1}^{N_h}(-A^s_h(r))^{\rho/2}S^h_s(r, \tau_2-r)\left(x, e^h_n(r)\right)e^h_n(r)\right\Vert^2dr\nonumber\\
&=&\sum_{n=1}^{N_h}\int_{\tau_1}^{\tau_2}\left(x, e^h_n(r)\right)^2\left(\lambda_n^h(r)\right)^{\rho}e^{-2\lambda^h_n(r)(\tau_2-r)}dr\nonumber\\
&\leq&\sum_{n=1}^{N_h}\int_{\tau_1}^{\tau_2}\sup_{0\leq r\leq T}\left(x, e^h_n(r)\right)^2\sup_{0\leq r\leq T}\left(\lambda^h_n(r)\right)^{\rho}e^{-2(\tau_2-r)\inf\limits_{0\leq r\leq T}\{\lambda^h_n(r)\}}dr\nonumber\\
&=&\sum_{n=1}^{N_h}\sup_{0\leq t\leq T}\left(x, e^h_n(t)\right)^2\sup_{0\leq t\leq T}\left(\lambda^h_n(t)\right)^{\rho}\int_{\tau_1}^{\tau_2}e^{-2(\tau_2-r)\inf\limits_{0\leq t\leq T}\{\lambda^h_n(t)\}}dr\nonumber\\
&=&\frac{1}{2}\sum_{n=1}^{N_h}\sup_{0\leq t\leq T}\left(x, e^h_n(t)\right)^2\sup_{0\leq t\leq T}\left(\lambda^h_n(t)\right)^{\rho}\left(\inf_{0\leq t\leq T}\lambda^h_n(t)\right)^{-1}\left(1-e^{-2\inf\limits_{0\leq t\leq T}\lambda_n^h(t)(\tau_2-\tau_1)}\right)\nonumber\\
&=&\frac{1}{2}\sum_{n=1}^{N_h}\sup_{0\leq t\leq T}\left(x, e^h_n(t)\right)^2\sup_{0\leq t\leq T}\left(\lambda^h_n(t)\right)^{\rho}\sup_{0\leq t\leq T}\left(\frac{1}{\lambda^h_n(t)}\right)\left(1-e^{-2\inf\limits_{0\leq t\leq T}\lambda_n^h(t)(\tau_2-\tau_1)}\right)\nonumber\\
&=&\frac{1}{2}\sum_{n=1}^{N_h}\sup_{0\leq t\leq T}\left(x, e^h_n(t)\right)^2\sup_{0\leq t\leq T}\left[\left(\lambda^h_n(t)\right)^{\rho}\left(\lambda^h_n(t)\right)^{-1}\right]\left(1-e^{-2\inf\limits_{0\leq t\leq T}\lambda_n^h(t)(\tau_2-\tau_1)}\right)\nonumber\\
&=&\frac{1}{2}\sum_{n=1}^{N_h}\sup_{0\leq t\leq T}\left(x, e^h_n(t)\right)^2\sup_{0\leq t\leq T}\left[\left(\lambda^h_n(t)\right)^{\rho-1}\right]\left(1-e^{-2\inf\limits_{0\leq t\leq T}\lambda_n^h(t)(\tau_2-\tau_1)}\right)\nonumber\\
&=&\frac{1}{2}\sum_{n=1}^{N_h}\sup_{0\leq t\leq T}\left(x, e^h_n(t)\right)^2\left(\inf_{0\leq t\leq T}\lambda^h_n(t)\right)^{\rho-1}\left(1-e^{-2\inf\limits_{0\leq t\leq T}\lambda_n^h(t)(\tau_2-\tau_1)}\right)
\end{eqnarray}
Using the boundedness of the function $x\longmapsto x^{\rho-1}(1-e^{-x})$ for $x\in[0, \infty)$, $\rho\in[0,1]$ and  the Parseval's identity, it follows from \eqref{sharp11} that 
\begin{eqnarray}
\label{sharp12}
\int_{\tau_1}^{\tau_2}\left\Vert(-A_h^s(r))^{\rho/2}S^h_s(r, \tau_2-r)P_hx\right\Vert^2&\leq& C(\tau_2-\tau_1)^{1-\rho}\sum_{n=1}^{N_h}\sup_{0\leq t\leq T}\left(x, e^h_n(t)\right)^2\nonumber\\
&=&C(\tau_2-\tau_1)^{1-\rho}\sup_{0\leq t\leq T}\sum_{n=1}^{N_h}\left(x, e^h_n(t)\right)^2\nonumber\\
&\leq& C(\tau_2-\tau_1)^{1-\rho}\sup_{0\leq t\leq T}\Vert x\Vert^2=C(\tau_2-\tau_1)^{1-\rho}\Vert x\Vert^2.
\end{eqnarray}
Substituting \eqref{sharp12} in \eqref{sharp11} completes the proof of \eqref{sharpa}. Let us now prove \eqref{sharpb}. Note that for $0\leq \rho<1$ the estimate \eqref{sharpb} follows from  \lemref{evolutionlemma}. The crucial case  is when $\rho=1$.
Note that the evolution parameter $U_h(\tau_2,r)$ satisfies the following integral equation, see e.g., \cite[Chapter 5]{Pazy}.
\begin{eqnarray}
\label{sharp1}
U_h(\tau_2,r)=S^h_r(\tau_2-r)+\int_r^{\tau_2}S^h_{\sigma}(\tau_2-\sigma)R^h(\sigma, r)d\sigma, 
\end{eqnarray}
where $R^h(\sigma, r)$ is defined as follows, see \cite[Chapter 5]{Pazy} 
\begin{eqnarray}
\label{sharp2}
R^h(\sigma, r)=\sum_{n=1}^{\infty}R^h_n(\sigma, r),
\end{eqnarray}
where $R^h_n(\sigma, r)$ satisfies the following recurrence relation, see e.g., \cite[Chapter 5]{Pazy} 
\begin{eqnarray}
\label{sharp3}
R^h_1(\sigma, r)=\left(A_h(r)-A_h(\sigma)\right)S^h_r(\sigma-r),\quad R^h_{n+1}(\sigma, r)=\int_r^{\sigma}R^h_1(\sigma, r)R^h_n(\gamma, r)d\gamma,\quad n\geq 1.
\end{eqnarray}
Using \eqref{sharp1}, the triangle inequality, the estimate $(a+b)^2\leq 2a^2+2b^2$, $a, b\in\mathbb{R}^+$  and \eqref{sharpa} yields
\begin{eqnarray}
\label{sharp4}
\int_{\tau_1}^{\tau_2}\Vert (-A_h(\tau))^{\rho/2}U_h(\tau_2,r)\Vert^2_{L(H)}ds&\leq& 2\int_{\tau_1}^{\tau_2}\Vert (-A_h(\tau))^{\rho/2}S^h_r(\tau_2-r)\Vert^2_{L(H)}dr\nonumber\\
&+&2\int_{\tau_1}^{\tau_2}\left\Vert\int_r^{\tau_2}(-A_h(\tau))^{\rho/2}S^h_{\sigma}(\tau_2-\sigma)R^h(\sigma, r)d\sigma\right\Vert^2_{L(H)}dr\nonumber\\
&=:&C(\tau_2-\tau_1)^{1-\rho}+ J_1.
\end{eqnarray}
  Using \lemref{evolutionlemma} and \eqref{smooth2} yields
\begin{eqnarray}
\label{sharp5}
J_1&\leq& \int_{\tau_1}^{\tau_2}\left(\int_r^{\tau_2}\Vert(-A_h(\tau))^{\rho/4}S^h_{\sigma}(\tau_2-\sigma)\Vert_{L(H)}\Vert (-A_h(\tau))^{\rho/4}R^h(\sigma, r)\Vert_{L(H)} d\sigma\right)^2dr\nonumber\\
&\leq& C\int_{\tau_1}^{\tau_2}\left(\int_{r}^{\tau_2}(\tau_2-\sigma)^{-
\rho/4}(\sigma-r)^{-\rho/4}d\sigma\right)^2dr\nonumber\\
&\leq& C\int_{\tau_1}^{\tau_2}\left(\int_r^{\frac{\tau_2+r}{2}}(\tau_2-\sigma)^{-\rho/4}(\sigma-r)^{-\rho/4}d\sigma\right)^2dr+C\int_{\tau_1}^{\tau_2}\left(\int^{\tau_2}_{\frac{\tau_2+r}{2}}(\tau_2-\sigma)^{-\rho/4}(\sigma-r)^{-\rho/4}d\sigma\right)^2dr\nonumber\\
&\leq& C\int_{\tau_1}^{\tau_2}(\tau_2-r)^{-\rho/2}\left(\int_r^{\frac{\tau_2+r}{2}}(\sigma-r)^{-\rho/4}d\sigma\right)^2dr+C\int_{\tau_1}^{\tau_2}(\tau_2-r)^{-\rho/2}\left(\int_{\frac{\tau_2+r}{2}}^{\tau_2}(\tau_2-\sigma)^{-\rho/4}d\sigma\right)^2dr\nonumber\\
&\leq& C\int_{\tau_1}^{\tau_2}(\tau_2-r)^{2-\rho}dr\nonumber\\
&\leq& C(\tau_2-\tau_1)^{3-\rho}.
\end{eqnarray}
Substituting  \eqref{sharp5} in \eqref{sharp4} yields 
\begin{eqnarray}
\int_{\tau_1}^{\tau_2}\left\Vert(-A_h(\tau))^{\rho/2}U_h(\tau_2, r)\right\Vert^2_{L(H)}dr\leq C(\tau_2-\tau_1)^{1-\rho}.
\end{eqnarray}
This completes the proof of \eqref{sharpb}. The proof of \eqref{sharpc} is similar to that of \eqref{sharpb}.
\end{proof}
\begin{lemma}
\label{sharpsecond}
Let $\rho\in[0, 1]$.  Under  \asssref{assumption2} and \ref{assumption5}, the following estimates hold
\begin{eqnarray}
\label{sharpd}
\int_{\tau_1}^{\tau_2}\left\Vert (-A_h(\tau))^{\rho}U_h(\tau_2, r)\right\Vert_{L(H)}dr&\leq& C(\tau_2-\tau_1)^{1-\rho},\quad 0\leq \tau_1\leq \tau_2\leq T,\quad \tau\in[0, T],\\
\label{sharpe}
\int_{\tau_1}^{\tau_2}\left\Vert U_h(\tau_2, r)(-A_h(\tau))^{\rho}\right\Vert_{L(H)}dr&\leq& C(\tau_2-\tau_1)^{1-\rho},\quad 0\leq \tau_1\leq \tau_2\leq T,\quad \tau\in[0,T].
\end{eqnarray}
\end{lemma}
\begin{proof}
We only prove \eqref{sharpd} since the proof of \eqref{sharpe} is similar.
Using \eqref{sharp1} and triangle inequality yields
\begin{eqnarray}
\label{vend1}
\int_{\tau_1}^{\tau_2}\left\Vert (-A_h(\tau))^{\rho}U_h(\tau_2, r)\right\Vert_{L(H)}dr&\leq& \int_{\tau_1}^{\tau_2}\left\Vert(-A_h(\sigma))^{\rho}S^h_r(\tau_2-r)\right\Vert_{L(H)}dr\nonumber\\
&+&\int_{\tau_1}^{\tau_2}\int_r^{\tau_2}\left\Vert (-A_h(\sigma))^{\rho}S^h_{\sigma}(\tau_2-\sigma)R^h(\sigma, r)\right\Vert_{L(H)}d\sigma dr\nonumber\\
&=:&J_2+J_3.
\end{eqnarray}
Using \lemref{evolutionlemma}, H\"{o}lder inequality and \eqref{sharpa} yields
\begin{eqnarray}
\label{vend2}
J_2&\leq& C\int_{\tau_1}^{\tau_2}\left\Vert(-A_h(r))^{\rho}S^h_r(\tau_2-r)\right\Vert_{L(H)}dr\nonumber\\
&=&C\int_{\tau_1}^{\tau_2}\left\Vert(-A_h(r))^{\rho/2}S^h_r\left(\frac{\tau_2-r}{2}\right)(-A_h(r))^{\rho/2}S^h_r\left(\frac{\tau_2-r}{2}\right)\right\Vert_{L(H)}dr\nonumber\\
&\leq&C\int_{\tau_1}^{\tau_2}\left\Vert(-A_h(r))^{\rho/2}S^h_r\left(\frac{\tau_2-r}{2}\right)\right\Vert^2_{L(H)}dr\nonumber\\
&\leq& C(\tau_2-\tau_1)^{1-\rho}.
\end{eqnarray}
Using  \lemref{evolutionlemma} yields
\begin{eqnarray}
\label{vend3}
J_3&\leq&\int_{\tau_1}^{\tau_2}\int_r^{\tau_2}\left\Vert(-A_h(\sigma))^{\rho/2}S^h_{\sigma}(\tau_2-\sigma)\right\Vert_{L(H)}\left\Vert(-A_h(\sigma))^{\rho/2}R^h(\sigma, r)\right\Vert_{L(H)}d\sigma dr\nonumber\\
&\leq& C\int_{\tau_1}^{\tau_2}\int_r^{\tau_2}(\tau_2-\sigma)^{-\rho/2}(\sigma-r)^{-\rho/2}d\sigma dr.
\end{eqnarray}
Splitting the second integral of \eqref{vend3} in two parts as in the estimate of $J_1$ \eqref{sharp5} and integrating yields
\begin{eqnarray}
\label{vend4}
J_3\leq C(\tau_2-\tau_1)^{2-\rho}.
\end{eqnarray}
Substituting \eqref{vend4} and \eqref{vend3} in \eqref{vend2} completes the proof of \eqref{sharpd}.
\end{proof}
The following space and time regularity hold for the semi-discrete problem \eqref{semi1}, and will be useful in our convergence analysis.
\begin{lemma} 
\label{lemma1}
Let \asssref{assumption2} (i) and (ii), \ref{assumption3} and \ref{assumption4} be fulfilled with
the corresponding  $0\leq \beta<2$. If $X_0\in L^p\left(\Omega, \mathcal{D}((-A(0))^{\beta/2})\right)$, $p\geq 2$, 
then for all $\gamma\in[0,\beta]$ the following regularity estimates hold
\begin{eqnarray}
\label{regular1}
\Vert (-A_h(0))^{\gamma/2}X^h(t)\Vert_{L^p(\Omega,H)}&\leq& C,\quad 0\leq t\leq T,\\
\label{regular2}
\Vert X^h(t_2)-X^h(t_1)\Vert_{L^p(\Omega,H)}&\leq& C(t_2-t_1)^{\min(\beta,1)/2}, \quad 0\leq t_1\leq t_2\leq T.
\end{eqnarray}
Moreover, if  \assref{assumption5} is fulfilled then the regularity estimates \eqref{regular1} and \eqref{regular2} still hold for $\beta=2$.
\end{lemma}
\begin{proof} The proof follows the same lines as that in \cite{Tambue3} for multiplicative noise. 
Note that in the case of additive noise, \eqref{regular1} shows that we can have a  spatial regularity estimate for $\gamma\in[0,2)$. 
Note that in the case of multiplicative noise \cite{Tambue3}, we can only take $\gamma\in[0,1)$. Note that the proof of  \lemref{lemma1} for $\beta=2$ makes use of \lemsref{sharp} and \ref{sharpsecond}. Note also that the optimal case \eqref{regular1} and \eqref{regular2} with $\beta=2$ are crucial to achieve optimal convergence order in time, which corresponds to  an analogue of the optimal regularity results in \cite{KruseLarsson}, for time independent self-adjoint operator $A$. 
\end{proof}
For non commutative operators $H_j$ on a Banach space, we introduce the following notation
\begin{eqnarray}
\prod_{j=l}^kH_j=\left\{\begin{array}{ll}
H_kH_{k-1}\cdots H_l,\quad \text{if} \quad k\geq l,\\
\mathbf{I},\quad \hspace{2cm} \text{if} \quad k<l.
\end{array}
\right.
\end{eqnarray}
The following lemma will be useful in our convergence proof.
\begin{lemma} 
\label{lemma2}
Let   \assref{assumption2} be fulfilled. Then the following estimate holds
\begin{eqnarray}
\label{comp1}
\left\Vert\left(\prod_{j=l}^me^{\Delta tA_{h,j}}\right)(-A_{h,l})^{\gamma}\right\Vert_{L(H)}&\leq& Ct_{m-l}^{-\gamma}, \quad 0\leq l< m,\quad 0\leq \gamma<1,\\
\label{comp1a}
\left\Vert(-A_{h,k})^{\gamma_1}\left(\prod_{j=l}^me^{\Delta tA_{h,j}}\right)(-A_{h,l})^{-\gamma_2}\right\Vert_{L(H)}&\leq& Ct_{m-l}^{\gamma_2-\gamma_1}, 
\quad 0\leq l< m,\; 0\leq k\leq M
\end{eqnarray}
$0\leq \gamma_1\leq 1$, $ 0<\gamma_2\leq 1$.
\end{lemma}
\begin{proof}
The proof can be found in \cite{Tambue2}.
\end{proof}

Let us consider the following deterministic problem : find $u\in V$ such that 
\begin{eqnarray}
\label{model3}
\frac{du}{dt}=A(t)u, \quad u(\tau)=v, \quad t\in(\tau,T].
\end{eqnarray}
The corresponding semi-discrete problem in space consists of finding $u^h\in V_h$ such that 
\begin{eqnarray}
\frac{du^h}{dt}=A_h(t)u_h, \quad u^h(\tau)=P_hv.
\end{eqnarray}

\begin{lemma}
\label{lemma3}
Let \assref{assumption2} be fulfilled. 
 For $v\in\mathcal{D}((-A(0))^{\alpha/2})$,  the following error estimate holds for the semi-discrete approximation of \eqref{model3} 
\begin{eqnarray}
\label{ErrorAntonio}
\Vert u(t)-u^h(t)\Vert\leq Ch^r(t-\tau)^{-(r-\alpha)/2}\Vert v\Vert_{\alpha},\quad 0\leq \alpha\leq r, \quad r\in[0,2].
\end{eqnarray}
\end{lemma}
\begin{proof}
The proof can be found in \cite{Tambue2}.
\end{proof}

\begin{lemma}
\label{spaceestimate}
Let  \asssref{assumption1}-\ref{assumption2}  be fulfilled. Let  $X(t)$ be the mild solution of \eqref{model1} and the  $X^h(t)$ the mild solution of \eqref{semi1}. 
\begin{itemize}
\item[(i)] If $0\leq \beta<2$, then the following error estimate holds
\begin{eqnarray}
\label{time1}
\Vert X(t)-X^h(t)\Vert_{L^2(\Omega,H)}\leq Ch^{\beta},\quad 0\leq t\leq T.
\end{eqnarray}
\item[(ii)] If $\beta=2$, then the following error estimate holds
\begin{eqnarray}
\Vert X(t)-X^h(t)\Vert_{L^2(\Omega,H)}\leq Ch^{2}\left[1+\max\left(\ln\left(\frac{t}{h^2}\right),0\right)\right],\quad 0< t\leq T.
\end{eqnarray}
\end{itemize}
\end{lemma}
\begin{proof} The proof  follows the same lines as the one in \cite{Tambue3} for multiplicative noise. 
\end{proof}

The proof of the following lemma can be found in \cite{Tambue3}.
\begin{lemma} Let  \assref{assumption2} be fulfilled.
\label{smoothing1}
 For any $\alpha\in[-1,1]$, the following estimate holds
\begin{eqnarray}
\left\Vert\left(U_h(t_j, t_{j-1})-e^{\Delta tA_{h,j-1}}\right)\left(-A_{h,j-1}\right)^{\alpha}\right\Vert_{L(H)}\leq C\Delta t^{1-\alpha}.
\end{eqnarray}
The following lemma can be found in \cite{Larsson2} or \cite{Tambue3}.
\end{lemma}
\begin{lemma}
\label{lemma9}
For all $\alpha_1,\alpha_2>0$ and $\alpha\in[0,1)$, there exist  two positive constants $C_{\alpha_1\alpha_2}$ and $C_{\alpha,\alpha_2}$ such that
\begin{eqnarray}
\label{son1}
\Delta t\sum_{j=1}^mt_{m-j+1}^{-1+\alpha_1}t_j^{-1+\alpha_2}\leq C_{\alpha_1\alpha_2}t_m^{-1+\alpha_1+\alpha_2},\quad
\Delta t\sum_{j=1}^mt_{m-j+1}^{-\alpha}t_j^{-1+\alpha_2}\leq C_{\alpha\alpha_2}t_m^{-\alpha+\alpha_2}.
\end{eqnarray}

\end{lemma}
\begin{lemma}
\label{fonda}Let  \assref{assumption2}  be fulfilled.
\begin{itemize}
\item[(i)] The following estimate holds for $ 1\leq i\leq m$
\begin{eqnarray}
\label{fonda0}
\left\Vert\left(\prod_{j=i}^mU_h(t_j,t_{j-1})\right)-\left(\prod_{j=i-1}^{m-1}
e^{\Delta tA_{h,j}}\right)\right\Vert_{L(H)}&\leq& C\Delta t^{1-\epsilon}.
\end{eqnarray}
\item[(ii)] The following estimate holds for $ 1\leq i\leq m$
\begin{eqnarray}
\label{fonda00}
\left\Vert\left[\left(\prod_{j=i}^mU_h(t_j,t_{j-1})\right)-\left(\prod_{j=i-1}^{m-1}
e^{\Delta tA_{h,j}}\right)\right](-A_h(0))^{-\epsilon}\right\Vert_{L(H)}&\leq& C\Delta t,
\end{eqnarray}
\item[(iii)]
 Then for all $1\leq i\leq m\leq M$.
 For all $\alpha\in[0, 1)$, the following estimate holds
\begin{eqnarray}
\label{fonda000}
\left\Vert\left[\left(\prod_{j=i}^mU_h(t_j,t_{j-1})\right)-\left(\prod_{j=i-1}^{m-1}
e^{\Delta tA_{h,j}}\right)\right](-A_{h,i-1})^{\alpha}\right\Vert_{L(H)}&\leq& C\Delta t^{1-\alpha-\epsilon}t_{m-i+1}^{-\alpha+\epsilon},
\end{eqnarray}
for an arbitrarily small $\epsilon>0$.
\end{itemize}
\end{lemma}
\begin{proof}
The proof of (i)-(ii) can be found in \cite{Tambue3}. The estimate \eqref{fonda000} is crucial to achieve higher convergence order in time on the terms 
involving the noise. Using the telescopic identity yields
\begin{eqnarray}
\label{fonda2}
\left(\prod_{j=i}^mU_h(t_j,t_{j-1})\right)-\left(\prod_{j=i}^{m}
e^{\Delta tA_{h,j-1}}\right)
&=&\left(\prod_{j=i+1}^mU_h(t_j, t_{j-1})\right)\left(U_h(t_i,t_{i-1})-e^{\Delta tA_{h,i-1}}\right)\nonumber\\
&+&\sum_{k=2}^{m-i+1}\left(\prod_{j=i+k}^mU_h(t_j,t_{j-1})\right)\left(U_h\left(t_{i+k-1}, t_{i+k-2}\right)-e^{\Delta tA_{h,i+k-2}}\right)\nonumber\\
&&.\left(\prod_{j=i}^{i+k-2}e^{\Delta tA_{h,j-1}}\right).
\end{eqnarray}
Taking the norm in both sides of \eqref{fonda2}, using  \lemsref{evolutionlemma},  \ref{smoothing1}, \ref{lemma2} and \ref{lemma9} yields
\begin{eqnarray}
\label{fonda7}
&&\left\Vert\left[\left(\prod_{j=i}^mU_h(t_j,t_{j-1})\right)-\left(\prod_{j=i}^{m}
e^{\Delta tA_{h,j-1}}\right)\right]\left(-A_{h,i-1}\right)^{\alpha}\right\Vert_{L(H)}\nonumber\\
&\leq& \left\Vert U_h(t_m,t_{i+k-1})\right\Vert_{L(H)}\left\Vert \left(U_h(t_{i},t_{ i-1})-e^{\Delta tA_{h,i-1}}\right)\left(-A_{h,i-1}\right)^{\alpha}\right\Vert_{L(H)}\nonumber\\
&+&\sum_{k=2}^{m-i+1}\left\Vert U_h(t_m, t_{i+k-1})(-A_{h,i})\right\Vert_{L(H)}\left\Vert(-A_{h,i})^{-1}\left(U_h(t_{i+k-1}, t_{i+k-2})-e^{\Delta tA_{h,i+k-2}}\right)\right\Vert_{L(H)}\nonumber\\
&\times&\left\Vert\left(\prod_{j=i}^{i+k-2}e^{\Delta tA_{h,j-1}}\right)(-A_{h,i-1})^{\alpha}\right\Vert_{L(H)}\nonumber\\
&\leq& C\Delta t^{1-\alpha}+C\sum_{k=2}^{m-i+1}t_{m-i-k}^{-1}\Delta t^2\,t_{k-1}^{-\alpha}\nonumber\\
&\leq& C\Delta t^{1-\alpha}+C\Delta t^{1-\alpha-\epsilon}\sum_{k=2}^{m-i+1}\Delta t \,t_{m-i-k}^{-1+\epsilon}t_{k-1}^{-\alpha}\nonumber\\
&\leq& C\Delta t^{1-\alpha-\epsilon}t_{m-i+1}^{-\alpha+\epsilon}.\nonumber\\
\end{eqnarray}
\end{proof}
\begin{lemma}
\label{ancien}
Under  \assref{assumption3} the following estimates hold
\begin{eqnarray}
\label{tamb1}
\Vert P_hF'(t,u)(u)v\Vert&\leq& C\Vert v\Vert, \quad\quad\quad u,v\in H,\quad t\in[0,T],\\
\label{tamb2}
\Vert (-A_h(\tau))^{\frac{-\eta}{2}}P_hF''(t,u)(v_1,v_2)\Vert&\leq& C\Vert v_1\Vert.\Vert v_2\Vert,\quad u, v_1,v_2\in H,\quad t,\tau\in [0,T],
\end{eqnarray}
where  $\eta \in [1,2)$ comes from  \assref{assumption3}. Note that the first and second order derivatives are taken respect the second variable.
\end{lemma}
\begin{proof}
The proof is a combination of  \lemref{lemma0} and \cite[Proposition 4.1]{TambueNg}.
\end{proof}
After the above preparatory results, we can now prove our main result.
\subsection{Proof of Theorem \ref{mainresult1}}
 Using triangle inequality, we split the fully discrete error in two parts as follows.
\begin{eqnarray}
\label{split}
\Vert X(t_m)-X^h_m\Vert_{L^2(\Omega, H)}\leq \Vert X(t_m)-X^h(t_m)\Vert_{L^2(\Omega,H)}+\Vert X^h(t_m)-X^h_m\Vert_{L^2(\Omega,H)}=: I+II.
\end{eqnarray}
The space error $I$ is estimated in  \lemref{spaceestimate}. It remains to estimate the time error $II$.  Note that the mild solution of \eqref{semi1} can be written as follows.
\begin{eqnarray}
\label{mild5}
X^h(t_m)=U_h(t_m,t_{m-1})X^h(t_{m-1})+\int_{t_{m-1}}^{t_m}U_h(t_m,s)P_hF\left(s,X^h(s)\right)ds
+\int_{t_{m-1}}^{t_m}U_h(t_m,s)P_hdW(s).
\end{eqnarray}
Iterating the mild solution \eqref{mild5} yields
\begin{eqnarray}
\label{mild6}
X^h(t_m)&=&\left(\prod_{j=1}^mU_h(t_j, t_{j-1})\right)P_hX_0+\int_{t_{m-1}}^{t_m}U_h(t_m,s)P_hF\left(s,X^h(s)\right)ds+\int_{t_{m-1}}^{t_m}U_h(t_m,s)P_hdW(s)\nonumber\\
&+&\sum_{k=1}^{m-1}\int_{t_{m-k-1}}^{t_{m-k}}\left(\prod_{j=m-k+1}^{m}U_h(t_j,t_{j-1})\right)U_h(t_{m-k},s)P_hF\left(s, X^h(s)\right)ds\nonumber\\
&+&\sum_{k=1}^{m-1}\int_{t_{m-k-1}}^{t_{m-k}}\left(\prod_{j=m-k+1}^{m}U_h(t_j,t_{j-1})\right)U_h(t_{m-k},s)P_hdW(s).
\end{eqnarray}
Iterating the numerical scheme \eqref{scheme4} by  substituting $X^h_j$, $j=m-1, \cdots,1$ only in the first term of \eqref{scheme4} by their expressions yields
\begin{eqnarray}
\label{num2}
X^h_m&=&\left(\prod_{j=0}^{m-1}e^{\Delta tA_{h,j}}\right)X^h_0+\int_{t_{m-1}}^{t_m}e^{(t_m-s)A_{h,m-1}}P_hF\left(t_{m-1}, X^h_{m-1}\right)ds+\int_{t_{m-1}}^{t_m}e^{\Delta tA_{h,m-1}}P_hdW(s)\nonumber\\
&+&\sum_{k=1}^{m-1}\int_{t_{m-k-1}}^{t_{m-k}}\left(\prod_{j=m-k}^{m-1}e^{\Delta tA_{h,j}}\right)e^{(t_{m-k}-s)A_{h,m-k-1}}P_hF\left(t_{m-k-1},X^h_{m-k-1}\right)ds\nonumber\\
&+&\sum_{k=1}^{m-1}\int_{t_{m-k-1}}^{t_{m-k}}\left(\prod_{j=m-k}^{m-1}e^{\Delta tA_{h,j}}\right)e^{\Delta tA_{h,m-k-1}}P_hdW(s).
\end{eqnarray}
Substracting \eqref{num2} from \eqref{mild6} yields
{\small
\begin{eqnarray}
\label{refait1}
X^h(t_m)-X^h_m&=&\left(\prod_{j=1}^mU_h(t_j, t_{j-1})\right)P_hX_0-\left(\prod_{j=0}^{m-1}e^{\Delta tA_{h,j}}\right)P_hX_0+\int_{t_{m-1}}^{t_m}U_h(t_m,s)P_hF\left(s,X^h(s)\right)ds\nonumber\\
&-&\int_{t_{m-1}}^{t_m}e^{(t_m-s)A_{h,m-1}}P_hF\left(t_{m-1}, X^h_{m-1}\right)ds+\int_{t_{m-1}}^{t_m}U_h(t_m,s)P_hdW(s)-\int_{t_{m-1}}^{t_m}e^{\Delta tA_{h,m-1}}P_hdW(s)\nonumber\\
&+&\sum_{k=1}^{m-1}\int_{t_{m-k-1}}^{t_{m-k}}\left(\prod_{j=m-k+1}^{m}U_h(t_j,t_{j-1})\right)U_h(t_{m-k},s)P_hF\left(s, X^h(s)\right)ds\nonumber\\
&-&\sum_{k=1}^{m-1}\int_{t_{m-k-1}}^{t_{m-k}}\left(\prod_{j=m-k}^{m-1}e^{\Delta tA_{h,j}}\right)e^{(t_{m-k}-s)A_{h,m-k-1}}P_hF\left(t_{m-k-1},X^h_{m-k-1}\right)ds\nonumber\\
&+&\sum_{k=1}^{m-1}\int_{t_{m-k-1}}^{t_{m-k}}\left(\prod_{j=m-k+1}^{m}U_h(t_j,t_{j-1})\right)U_h(t_{m-k},s)P_hdW(s)\nonumber\\
&-&\sum_{k=1}^{m-1}\int_{t_{m-k-1}}^{t_{m-k}}\left(\prod_{j=m-k}^{m-1}e^{\Delta tA_{h,j}}\right)e^{\Delta tA_{h,m-k-1}}P_hdW(s)\nonumber\\
&=:&II_1+II_2+II_3+II_4+II_5.
\end{eqnarray}
}
Taking the norm in both sides of \eqref{refait1} yields 
\begin{eqnarray}
\label{refait2}
\Vert X^h(t_m)-X^h_m\Vert^2_{L^2(\Omega,H)}\leq 25\sum_{i=1}^5\Vert II_i\Vert^2_{L^2(\Omega,H)}.
\end{eqnarray}
We estimate separately $\Vert II_i\Vert_{L^2(\Omega,H)}$, $i=1,\cdots, 5$.

\subsubsection{Estimate of $II_1$, $II_2$ and $II_3$}
 Using  \lemref{fonda} (ii), it holds that
 {\small
 \begin{eqnarray}
 \label{multi1}
 \Vert II_1\Vert_{L^2(\Omega,H)}&\leq& \left\Vert \left[\left(\prod_{j=1}^mU_h(t_j, t_{j-1})\right)-\left(\prod_{j=0}^{m-1}e^{\Delta tA_{h,j}}\right)\right](-A_h(0))^{-\beta/2}\right\Vert_{L(H)}\Vert (-A_h(0))^{\beta/2}X_0\Vert_{L^2(\Omega,H)}\nonumber\\
 &\leq& C\Delta t\Vert -A(0))^{\beta/2}X_0\Vert_{L^2(\Omega, H)}\leq C\Delta t.
 \end{eqnarray}
 }
 Similarly to \cite{Tambue3}, we have the following estimate
 \begin{eqnarray}
 \label{multi2}
 \Vert II_2\Vert_{L^2(\Omega,H)}\leq C\Delta t+C\Delta t\Vert X^h(t_{m-1})-X^h_{m-1}\Vert_{L^2(\Omega,H)}.
 \end{eqnarray}
To estimate $II_3$, we split it in two terms as follows
 \begin{eqnarray}
 \label{man1}
 II_3&=&\int_{t_{m-1}}^{t_m}\left(U_h(t_m, s)-U_h(t_m, t_{m-1})\right)P_hdW(s)+\int_{t_{m-1}}^{t_m}\left(U_h(t_m, t_{m-1})-e^{\Delta tA_{h,m-1}}\right)dW(s)\nonumber\\
 &=:&II_{31}+II_{32}.
 \end{eqnarray}
 Applying the It\^{o}-isometry property, using   \lemsref{lemma1a} and \ref{evolutionlemma} yields
 {\small
 \begin{eqnarray}
 \label{man2}
 \Vert II_{31}\Vert^2_{L^2(\Omega,H)}
 &=&\int_{t_{m-1}}^{t_m}\mathbb{E}\left\Vert\left(
 U_h(t_m,s)-U_h(t_m, t_{m-1})\right)P_hQ^{\frac{1}{2}}\right\Vert^2_{\mathcal{L}_2(H)}ds\nonumber\\
 &\leq& \int_{t_{m-1}}^{t_m}\left\Vert U_h(t_m,s)\left(\mathbf{I}-U_h(s,t_{m-1})\right)\left(-A_{h,m}\right)^{\frac{1-\beta}{2}}\right\Vert^2_{L(H)}\left\Vert\left(-A_{h,m}\right)^{\frac{\beta-1}{2}}P_hQ^{\frac{1}{2}}\right\Vert^2_{\mathcal{L}_2(H)}ds\nonumber\\
 &\leq& C\int_{t_{m-1}}^{t_m}\left\Vert U_h(t_m,s)\left(-A_{h,m}\right)^{\frac{1-\epsilon}{2}}\right\Vert^2_{L(H)}\left\Vert\left(-A_{h,m}\right)^{\frac{-1+\epsilon}{2}}\left(\mathbf{I}-U_h(s,t_{m-1})\right)\left(-A_{h,m}\right)^{\frac{1-\beta}{2}}\right\Vert^2_{L(H)}ds\nonumber\\
 &\leq& C\int_{t_{m-1}}^{t_m}(t_m-s)^{-1+\epsilon}(s-t_{m-1})^{\beta-\epsilon}ds\nonumber\\
 &\leq& C\Delta t^{\beta-\epsilon}\int_{t_{m-1}}^{t_m}(t_m-s)^{-1+\epsilon}ds\leq  C\Delta t^{\beta}.\nonumber\\
 \end{eqnarray}
 }
 Applying the It\^{o}-isometry property, using  \lemsref{lemma1a}, \ref{fonda} (ii) (if $\beta\geq 1 $),  \lemref{fonda} (iii) with $\alpha=\frac{1-\beta}{2}$ (if $\beta<1$)  yields
 \begin{eqnarray}
 \label{man3}
\Vert II_{32}\Vert^2_{L^2(\Omega,H)}
&=&\int_{t_{m-1}}^{t_m}\mathbb{E}\left\Vert\left(U_h(t_m, t_{m-1})-e^{\Delta tA_{h, m-1}}\right)P_hQ^{\frac{1}{2}}\right\Vert^2_{\mathcal{L}_2(H)}ds\nonumber\\
 &\leq& \int_{t_{m-1}}^{t_m}\mathbb{E}\left\Vert\left(U_h(t_m, t_{m-1})-e^{\Delta tA_{h,m-1}}\right)\left(-A_{h,m-1}\right)^{\frac{1-\beta}{2}}\right\Vert^2_{L(H)}\left\Vert \left(-A_{h,m-1}\right)^{\frac{\beta-1}{2}}P_hQ^{\frac{1}{2}}\right\Vert^2_{\mathcal{L}_2(H)}ds\nonumber\\
 &\leq& C\int_{t_{m-1}}^{t_m}\Delta t^{1+\beta}ds\leq C\Delta t^{2+\beta}.\nonumber\\
 \end{eqnarray}
 Substituting \eqref{man3} and \eqref{man2} in \eqref{man1} yields
 \begin{eqnarray}
 \label{addiv1}
 \Vert II_3\Vert^2_{L^2(\Omega,H)}\leq C\Delta t^{\beta}.
 \end{eqnarray}
  \subsubsection{Estimate of $II_4$}
 \label{TalyorBanach}
To estimate $II_4$, we split it in five terms as follows.
\begin{eqnarray}
\label{eza1}
II_{4}
&=&\sum_{k=1}^{m-1}\int_{t_{m-k-1}}^{t_{m-k}}\left(\prod_{j=m-k+1}^mU_h(t_j,t_{j-1})\right)U_h(t_{m-k},s)\left[P_hF\left(s,X^h(s)\right)-P_hF\left(t_{m-k-1},X^h(t_{m-k-1})\right)\right]ds\nonumber\\
&+&\sum_{k=1}^{m-1}\int_{t_{m-k-1}}^{t_{m-k}}\left(\prod_{j=m-k+1}^mU_h(t_j,t_{j-1})\right)\left[U_h(t_{m-k},s)-U_h(t_{m-k}, t_{m-k-1})\right]P_hF\left(t_{m-k-1},X^h(t_{m-k-1})\right)ds\nonumber\\
&+&\sum_{k=1}^{m-1}\int_{t_{m-k-1}}^{t_{m-k}}\left[\left(\prod_{j=m-k}^mU_h(t_j,t_{j-1})\right)-\left(\prod_{j=m-k-1}^{m-1}e^{\Delta tA_{h,j}}\right)\right]P_hF\left(t_{m-k-1},X^h(t_{m-k-1})\right)ds\nonumber\\
&+&\sum_{k=1}^{m-1}\int_{t_{m-k-1}}^{t_{m-k}}\left(\prod_{j=m-k}^{m-1}e^{\Delta tA_{h,j}}\right)\left(e^{\Delta tA_{h,m-k-1}}-e^{(t_{m-k}-s)A_{h,m-k-1}}\right)P_hF\left(t_{m-k-1},X^h(t_{m-k-1})\right)ds\nonumber\\
&+&\sum_{k=1}^{m-1}\int_{t_{m-k-1}}^{t_{m-k}}\left(\prod_{j=m-k}^{m-1}e^{\Delta tA_{h,j}}\right)e^{(t_{m-k}-s)A_{h,m-k-1}}\left[P_hF\left(t_{m-k-1},X^h(t_{m-k-1})\right)-P_hF\left(t_{m-k-1},X^h_{m-k-1}\right)\right]ds\nonumber\\
&=:&II_{41}+II_{42}+II_{43}+II_{44}+II_{45}.
\end{eqnarray}
Similarly to \cite{Tambue3}, we have the following estimate
\begin{eqnarray}
\label{eza3}
\Vert II_{42}\Vert_{L^2(\Omega,H)}+ \Vert II_{43}\Vert_{L^2(\Omega,H)}+ \Vert II_{44}\Vert_{L^2(\Omega,H)}\leq C\Delta t.
\end{eqnarray}
It remains to estimate $\Vert II_{41}\Vert_{L^2(\Omega, H)}$ and $\Vert II_{45}\Vert_{L^2(\Omega, H)}$. Let us start with the estimate of $\Vert II_{41}\Vert_{L^2(\Omega, H)}$ as it is easy.
Using  \lemref{lemma2} and  \assref{assumption3} yields
\begin{eqnarray}
\label{eza6}
\Vert II_{45}\Vert_{L^2(\Omega,H)}\leq C\sum_{k=1}^{m-1}\int_{t_{m-k-1}}^{t_{m-k}}\Vert X^h\left(t_{m-k-1}\right)-X^h_{m-k-1}\Vert_{L^2(\Omega, H)}\leq C\Delta t\sum_{k=0}^{m-1}\Vert X^h(t_k)-X^h_k\Vert_{L^2(\Omega,H)}.
\end{eqnarray}
To estimate $II_{41}$, we decompose it in two terms as follows
 \begin{eqnarray}
 \label{do1}
 II_{41}&=&\sum_{k=1}^{m-1}\int_{t_{m-k-1}}^{t_{m-k}}U_h(t_m,s)\left[P_hF\left(s,X^h(s)\right)-P_hF\left(t_{m-k-1}, X^h(s)\right)\right]ds\nonumber\\
 &+&\sum_{k=1}^{m-1}\int_{t_{m-k-1}}^{t_{m-k}}U_h(t_m,s)\left[P_hF\left(t_{m-k-1}, X^h(s)\right)-P_hF\left(t_{m-k-1}, X^h(t_{m-k-1})\right)\right]ds\nonumber\\
 &=:&II_{411}+II_{412}.
 \end{eqnarray}
 Using   \assref{assumption3} and  \lemref{evolutionlemma}  yields
 \begin{eqnarray}
 \label{do2}
 \Vert II_{411}\Vert_{L^2(\Omega,H)}\leq C\Delta t^{\beta/2}.
 \end{eqnarray}
 To achieve higher order in $II_{412}$, we apply Taylor's formula in Banach space to $F$. This yields
 \begin{eqnarray}
 \label{mach1}
II_{412}
 &=&\sum_{k=1}^{m-1}\int_{t_{m-k-1}}^{t_{m-k}}U_h(t_m,s)P_hF'\left(t_{m-k-1}, X^h(t_{m-k-1})\right)\left(U_h(s, t_{m-k-1})-\mathbf{I}\right)X^h(t_{m-k-1})ds\nonumber\\
 &+&\sum_{k=1}^{m-1}\int_{t_{m-k-1}}^{t_{m-k}}U_h(t_m,s)P_hF'\left(t_{m-k-1}, X^h(t_{m-k-1})\right)\int_{t_{m-k-1}}^sU_h(s,\sigma)P_hF\left(\sigma, X^h(\sigma)\right)d\sigma ds\nonumber\\
 &+&\sum_{k=1}^{m-1}\int_{t_{m-k-1}}^{t_{m-k}}U_h(t_m,s)P_hF'\left(t_{m-k-1}, X^h(t_{m-k-1})\right)\int_{t_{m-k-1}}^sU_h(s,\sigma)P_hdW(\sigma)ds\nonumber\\
 &+&\sum_{k=1}^{m-1}\int_{t_{m-k-1}}^{t_{m-k}}U_h(t_m, s)\chi^hds\nonumber\\
 &=:&II_{412}^{(1)}+II_{412}^{(2)}+II_{412}^{(3)}+II_{412}^{(4)}, \nonumber\\
 \end{eqnarray}
 where
 \begin{eqnarray}
 \label{chi}
 \chi^h&=&\int_0^1P_hF''\left(t_{m-k-1}, X^h(t_{m-k-1})+\lambda\left(X^h(s)-X^h(t_{m-k-1})\right)\right)\nonumber\\
 &&.\left(X^h(s)-X^h(t_{m-k-1}), X^h(s)-X^h(t_{m-k-1})\right)(1-\lambda)d\lambda.
 \end{eqnarray}
 Using   \lemsref{evolutionlemma}, \ref{ancien}, \ref{lemma1} and  \ref{sharpsecond} yields
 \begin{eqnarray}
 \label{regarde1}
 \Vert II_{412}^{(1)}\Vert_{L^2(\Omega,H)}&\leq& C\sum_{k=1}^{m-1}\int_{t_{m-k-1}}^{t_{m-k}}\left\Vert \left(U_h(s, t_{m-k-1})-\mathbf{I}\right)\left(-A_{h,m-k-1}\right)^{-\beta/2}\right\Vert_{L(H)}\nonumber\\
 &\times&\left\Vert \left(-A_{h,m-k-1}\right)^{\beta/2}X^h(t_{m-k-1})\right\Vert_{L^2(\Omega,H)}ds\nonumber\\
 &\leq& C\sum_{k=1}^{m-1}\int_{t_{m-k-1}}^{t_{m-k}}(s-t_{m-k-1})^{\beta/2}ds\leq C\Delta t^{\beta/2}.
 \end{eqnarray}
 Using  \lemref{evolutionlemma}, \assref{assumption3}, \lemsref{lemma1a} and \ref{ancien}  yields
 \begin{eqnarray}
 \label{mach3}
 \Vert II_{412}^{(2)}\Vert_{L^2(\Omega,H)}\leq C\sum_{k=1}^{m-1}\int_{t_{m-k-1}}^{t_{m-k}}\int_{t_{m-k-1}}^sd\sigma ds\leq C\Delta t.
 \end{eqnarray}
 Applying the It\^{o}-isometry property,  using the fact that the expectation of the cross-product vanishes,  H\"{o}lder inequality, \lemsref{ancien},  \ref{lemma1a} and \ref{evolutionlemma} yields
 {\small
 \begin{eqnarray}
 \label{mach4}
 \Vert II_{412}^{(3)}\Vert^2_{L^2(\Omega, H)}
 &=& \sum_{k=1}^{m-1}\mathbb{E}\left[\left\Vert\int_{t_{m-k-1}}^{t_{m-k}}U_h(t_m,s)P_hF'\left(t_{m-k-1}, X^h(t_{m-k-1})\right)\int_{t_{m-k-1}}^sU_h(s,\sigma)P_hdW(\sigma)ds\right\Vert^2  \right]\nonumber\\
 &\leq&\Delta t\sum_{k=1}^{m-1}\int_{t_{m-k-1}}^{t_{m-k}}\int_{t_{m-k-1}}^s\mathbb{E}\left\Vert U_h(t_m,s)P_hF'\left(t_{m-k-1}, X^h(t_{m-k-1})\right)\right\Vert^2_{L(H)}\nonumber\\
 &&\times\left\Vert U_h(s,\sigma)P_hQ^{\frac{1}{2}}\right\Vert^2_{\mathcal{L}_2(H)}d\sigma ds\nonumber\\
 &\leq & C\Delta t\sum_{k=1}^{m-1}\int_{t_{m-k-1}}^{t_{m-k}}\int_{t_{m-k-1}}^s\left\Vert U_h(s,\sigma)(-A_h(\sigma))^{\frac{1-\beta}{2}}\right\Vert^2_{L(H)}\left\Vert (-A_h(\sigma))^{\frac{\beta-1}{2}}P_hQ^{\frac{1}{2}}\right\Vert^2_{\mathcal{L}_2(H)}d\sigma ds\nonumber\\
 &\leq& C\Delta t\sum_{k=1}^{m-1}\int_{t_{m-k-1}}^{t_{m-k}}\int_{t_{m-k-1}}^s(s-\sigma)^{\min(-1+\beta,0)}d\sigma ds\leq C\Delta t^{\min(1+\beta, 2)}.\nonumber\\
 \end{eqnarray}
 }
 Using  \lemsref{ancien} and \ref{lemma1}, it follows from \eqref{chi} that
 \begin{eqnarray}
 \label{estichi}
 \left\Vert\left(-A_{h,m-k-1}\right)^{-\frac{\eta}{2}}\chi^h\right\Vert_{L^2(\Omega,H)}&\leq& C\left\Vert\Vert X^h(s)-X^h(t_{m-k-1})\Vert^2\right\Vert_{L^2(\Omega,H)}\leq C\Delta t^{\min(\beta,1)}.
 \end{eqnarray}
 Hence, from \eqref{mach1}, using \eqref{estichi}, \lemref{evolutionlemma} we have 
 \begin{eqnarray}
 \label{mach5}
 \Vert II_{412}^{(4)}\Vert_{L^2(\Omega,H)}
 &\leq& \sum_{k=1}^{m-1}\int_{t_{m-k-1}}^{t_{m-k}}\left\Vert U_h(t_m,s)\left(-A_{h,m-k-1}\right)^{\frac{\eta}{2}}\right\Vert_{L(H)}\left\Vert\left(-A_{h,m-k-1}\right)^{-\frac{\eta}{2}}\chi^h\right\Vert_{L^2(\Omega,H)}ds\nonumber\\
 &\leq&C\Delta t^{\min(\beta,1)}\sum_{k=1}^{m-1}\int_{t_{m-k-1}}^{t_{m-k}}(t_m-s)^{-\frac{\eta}{2}}ds\nonumber\\
 &\leq& C\Delta t^{\min(\beta,1)}\int_0^{t_m}(t_m-s)^{-\frac{\eta}{2}}ds\leq C\Delta t^{\min(\beta,1)}.
 \end{eqnarray}
 Substituting \eqref{mach5}, \eqref{mach4}, \eqref{mach3} and \eqref{regarde1}  in \eqref{mach1} yields
 \begin{eqnarray}
 \label{do3}
 \Vert II_{412}\Vert_{L^2(\Omega,H)}\leq C\Delta t^{\beta/2}.
 \end{eqnarray}
 Substituting \eqref{do3} and \eqref{do2} in \eqref{do1} yields
 \begin{eqnarray}
 \label{banc3}
 \Vert II_{41}\Vert_{L^2(\Omega,H)}\leq C\Delta t^{\beta/2}.
 \end{eqnarray}
 Substituting \eqref{banc3}, \eqref{eza3} and \eqref{eza6} in \eqref{eza1} yields
 \begin{eqnarray}
 \label{bruit2}
 \Vert II_4\Vert_{L^2(\Omega,H)}\leq C\Delta t^{\beta/2}+C\Delta t\sum_{k=0}^{m-1}\Vert X^h(t_k)-X^h_k\Vert_{L^2(\Omega,H)}.
 \end{eqnarray}
 
 \subsubsection{Estimate of  $II_5$}
 \label{Noiseestimate}
 To estimate $II_5$, we split it into two terms as follows
 \begin{eqnarray}
 \label{man4}
 II_5&=&\sum_{k=1}^{m-1}\int_{t_{m-k-1}}^{t_{m-k}}\left(\prod_{j=m-k+1}^mU_h(t_j,t_{j-1})\right)\left[U_h(t_{m-k},s)-U_h(t_{m-k}, t_{m-k-1})\right]P_hdW(s)\nonumber\\
 &+&\sum_{k=1}^{m-1}\int_{t_{m-k-1}}^{t_{m-k}}\left[\left(\prod_{j=m-k}^mU_h(t_j,t_{j-1})\right)-\left(\prod_{j=m-k-1}^{m-1}e^{\Delta tA_{h,j}}\right)\right]P_hdW(s)\nonumber\\
 &=:&VI_{51}+VI_{52}.
 \end{eqnarray}
 Applying the It\^{o}-isometry property, using  \lemsref{evolutionlemma}, \ref{lemma1a} and \ref{sharpsecond} yields
 \begin{eqnarray}
 \label{man5}
&&\Vert II_{51}\Vert^2_{L^2(\Omega,H)}\nonumber\\
 &=&\sum_{k=1}^{m-1}\int_{t_{m-k-1}}^{t_{m-k}}\mathbb{E}\left\Vert\left(\prod_{j=m-k+1}^mU_h(t_j, t_{j-1})\right)U_h(t_{m-k}, s)\left(\mathbf{I}-U_h(s, t_{m-k-1})\right)P_hQ^{\frac{1}{2}}\right\Vert^2_{\mathcal{L}_2(H)}ds\nonumber\\
 &\leq& \sum_{k=1}^{m-1}\int_{t_{m-k-1}}^{t_{m-k}}\left\Vert U_h(t_m, s)(-A_{h,m-k-1})^{\frac{1}{2}}\right\Vert^2_{L(H)}\nonumber\\
 &\times&\left\Vert(-A_{h,m-k-1})^{\frac{-1}{2}}\left(\mathbf{I}-U_h(s, t_{m-k-1})\right)(-A_{h,m-k-1})^{\frac{1-\beta}{2}}\right\Vert^2_{L(H)}\left\Vert (-A_{h,m-k-1})^{\frac{\beta-1}{2}}P_hQ^{\frac{1}{2}}\right\Vert^2_{\mathcal{L}_2(H)}ds\nonumber\\
 &\leq& C\sum_{k=1}^{m-1}\int_{t_{m-k-1}}^{t_{m-k}}\left\Vert U_h(t_m, s)(-A_h(0))^{\frac{1}{2}}\right\Vert^2_{L(H)}(s-t_{m-k-1})^{\beta}ds\nonumber\\
 &\leq& C\Delta t^{\beta}\sum_{k=1}^{m-1}\int_{t_{m-k-1}}^{t_{m-k}}\left\Vert U_h(t_m, s)(-A_h(0))^{\frac{1}{2}}\right\Vert^2_{L(H)}ds\nonumber\\
 &\leq& C\Delta t^{\beta}\int_0^{t_m}\left\Vert U_h(t_m, s)(-A_h(0))^{\frac{1}{2}}\right\Vert^2_{L(H)}ds\nonumber\\
 &\leq& C\Delta t^{\beta}.\nonumber\\
 \end{eqnarray}
 Applying the It\^{o}-isometry property, using  \lemsref{fonda} (iii) and  \ref{lemma1a} yields
 \begin{eqnarray}
 \label{man6}
 \Vert II_{52}\Vert^2_{L^2(\Omega,H)}
 &=&\sum_{k=1}^{m-1}\int_{t_{m-k-1}}^{t_{m-k}}\left\Vert\left[\left(\prod_{j=m-k}^mU_h(t_j, t_{j-1})\right)-\left(\prod_{j=m-k-1}^{m-1}e^{\Delta tA_{h,j}}\right)\right]P_hQ^{\frac{1}{2}}\right\Vert^2_{\mathcal{L}_2(H)}ds\nonumber\\
 &\leq&\sum_{k=1}^{m-1}\int_{t_{m-k-1}}^{t_{m-k}}\left\Vert\left[\left(\prod_{j=m-k}^mU_h(t_j, t_{j-1})\right)-\left(\prod_{j=m-k-1}^{m-1}e^{\Delta tA_{h,j}}\right)\right]\left(-A_{h,m-k-1}\right)^{\frac{1-\beta}{2}}\right\Vert^2_{L(H)}\nonumber\\
 &\times&\left\Vert\left(-A_{h,m-k-1}\right)^{\frac{\beta-1}{2}} P_hQ^{\frac{1}{2}}\right\Vert^2_{\mathcal{L}_2(H)}ds\nonumber\\
 &\leq& C\Delta t^{1+\beta-\epsilon}\sum_{k=2}^{m-1}\int_{t_{m-k-1}}^{t_{m-k}}t_{k+1}^{-1+\beta+\epsilon}ds\leq  C\Delta t^{\beta}.\nonumber\\
 \end{eqnarray}
 Substituting \eqref{man6} and \eqref{man5} in \eqref{man4} yields
 \begin{eqnarray}
 \label{addi2}
 \Vert II_5\Vert^2_{L^2(\Omega,H)}\leq C\Delta t^{\beta}.
 \end{eqnarray}
 Substituting \eqref{addi2}, \eqref{bruit2}, \eqref{multi1} and \eqref{multi2} in \eqref{refait1} yields
 \begin{eqnarray}
 \label{bruit3}
 \Vert X^h(t_m)-X^h_m\Vert^2_{L^2(\Omega,H)}\leq C\Delta t^{\beta}+C\Delta t\sum_{k=0}^{m-1}\Vert X^h(t_k)-X^h_k\Vert^2_{L^2(\Omega,H)}.
 \end{eqnarray}
 Applying the discrete Gronwall's lemma to \eqref{bruit3}  completes the proof of  \thmref{mainresult1}.
\section{Numerical experiments}
\label{experiment} 
We consider the reaction diffusion equation 
\begin{eqnarray}
\label{linear}
  dX=[D(t) \varDelta X -k(t) X]dt+ dW 
\qquad \text{given } \quad X(0)=X_{0}=0,
\end{eqnarray}
in the time interval $[0,T]$ with diffusion coefficient $ D(t)=(1/10)(1+e^{-t})$ and reaction rate $k(t)=1$
on homogeneous Neumann boundary conditions on the domain
$\Lambda=[0,L_{1}]\times [0,L_{2}]$. We take $L_1=L_2=1$.
Our function $F(t,u)=k(t)u $ is linear and obviously satisfies  \assref{assumption3}. 
Since $F(t, u)$ is linear on the second variable, it holds that $F'(t,u)v=k(t)v$ for all $u, v\in L^2(\Lambda)$, where $F'$ stands
for the differential  with respect to the second variable. 
Therefore $\Vert F'(t,u)\Vert_{L(H)}\leq\vert k(t)\vert =1$ for all $u\in L^2(\Lambda)$. Obviously we have $F''(t,u)=0$, for all $u\in L^2(\Lambda)$.
In general, we are interested in nonlinear $F$ however for this linear
system we can find a good approximation of the  exact solution to compare our numerics to.
The eigenfunctions $\{e_{i}^{(1)}e_{j}^{(2)}\}_{i,j\geq 0}
$ of the operator $-\varDelta$ here are given by 
\begin{eqnarray}
e_{0}^{(l)}=\sqrt{\dfrac{1}{L_{l}}},\;\;\;\lambda_{0}^{(l)}=0,\;\;\;
e_{i}^{(l)}=\sqrt{\dfrac{2}{L_{l}}}\cos(\lambda_{i}^{(l)}x),\;\;\;\lambda_{i}^{(l)}=\dfrac{i
  \,\pi }{L_{l}},
\end{eqnarray}
where $l \in \left\lbrace 1, 2 \right\rbrace$ and  $i=\{1, 2, 3, \cdots\}$
with the corresponding eigenvalues $ \{\lambda_{i,j}\}_{i,j\geq 0} $ given by 
$\lambda_{i,j}= (\lambda_{i}^{(1)})^{2}+ (\lambda_{j}^{(2)})^{2}.$
The linear operator is $A(t)=D(t)\varDelta$ and has the same eigenfunctions as $\varDelta$, but  with eigenvalues $\{D(t)\lambda_{i,j}\}_{i,j\geq 0}$.
Clearly we have $\mathcal{D}(A(t))=H^2(\Lambda)$ and $\mathcal{D}((A(t))^{\alpha})=\mathcal{D}((A(0))^{\alpha})$ for all $t\in[0,T]$ and $0\leq \alpha\leq 1$.
Since $D(t)$ is bounded below by $(1/10)(1+e^{-T})$, it follows that the ellipticity condition \eqref{ellip} and therefore as a consequence of 
the analysis in  \secref{fullyscheme}, it follows that $A(t)$ are uniformly sectorial.
%
Obviously \assref{assumption2}  and  \eqref{supcond} are also fulfilled.
We also used
\begin{eqnarray}
\label{noise2}
 q_{i,j}=\left( i^{2}+j^{2}\right)^{-(\beta +\delta)}, \, \beta>0,
\end{eqnarray} 
in the representation \eqref{noise2} for some small $\delta>0$. Here
the noise and the linear operator are supposed to have the same
eigenfunctions. We obviously have 
\begin{eqnarray}
\underset{(i,j) \in \mathbb{N}^{2}}{\sum}\lambda_{i,j}^{\beta-1}q_{i,j}<  \pi^{2}\underset{(i,j) 
\in \mathbb{N}^{2}}{\sum} \left( i^{2}+j^{2}\right)^{-(1+\delta)} <\infty,
\end{eqnarray}
thus  \assref{assumption4} 
is satisfied.  In  our simulations, we take
$\beta\in \{1.5,2\}$, with $\delta=0.001$. The  close form  of  the exact solution of \eqref{linear} is known. Indeed using the representation of the noise in \eqref{noise2},
the decomposition of \eqref{linear} in each eigenvector node yields
the following  Ornstein-Uhlenbeck process 
\begin{eqnarray}
\label{exact}
 dX_{i}=-(D(t) \lambda_{i}+k(t))X_{i}dt+ \sqrt{q_{i}}d\beta_{i}(t)\qquad i \in \mathbb{N}^{2}.
\end{eqnarray}
This is a Gaussian process with the mild solution
\begin{eqnarray}
X_{i}(t)= e^{- \int_{0}^{t}b_{i}(s)ds} \left[ X_{i}(0)+  \sqrt{q_{i}}\int_{0}^{t}e^{ \int_{0}^{s} b_{i}(y)dy} d \beta_{i}(s)\right],\; b_{i}(t)=D(t) \lambda_{i}+k(t).
\end{eqnarray}
Applying the Ito isometry yields the  following variance of $X_{i}(t)$
\begin{eqnarray}
\label{exact2}
 \text{Var}(X_{i}(t))=q_{i}\, e^{- \int_{0}^{t}\,2\,b_{i}(s)ds}  \left( \int_{0}^{t}e^{\int_{0}^{s} \,2\,b_{i}(y)dy} ds \right).
\end{eqnarray}
During simulation, we compute the exact solution recurrently as 
\begin{eqnarray}
\label{exact1}
  X_{i}^{m+1}  = e^{- \int_{t_m}^{t_{m+1}}b_{i}(s)ds} X_{i}^m +\left(q_{i}\, e^{- \int_{t_m}^{t_{m+1}}\,2\,b_{i}(s)ds}  \left( \int_{t_m}^{t_{m+1}}e^{\int_{t_m}^{s} \,2\,b_{i}(y)dy} ds \right) \right)^{1/2}R_{i,m},\nonumber \\
  \end{eqnarray}
where $R_{i,m}$ are independent, standard normally distributed random
variables with mean $0$ and variance $1$.
Note that  the integrals involved in \eqref{exact1} are computed exactly  for the first integral and accurately appoximated for the second  integral.

\begin{figure}[!ht]
\begin{center}
 \includegraphics[width=0.45\textwidth]{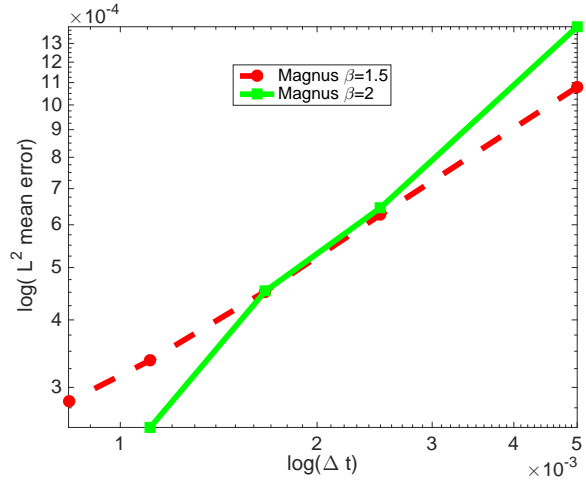}
  \end{center}
 \caption{Convergence of the  stochastic Magnus scheme for $\beta=1.5$ and $\beta=2$ in \eqref{noise2}.
 The order of convergence in time  is $0.995$  for $\beta=2$, and $0.7561$ for $\beta=1.5$. The total number of samples used is $100$.}
 \label{FIGI}
 \end{figure}
 In Figure \ref{FIGI}, we can observe the convergence of the stochastic Magnus scheme  for two noise's parameters.
 Indeed the order of convergence in time is $0.995$  for $\beta=2$, and $0.7561$ for $\beta=1.5$. 
 These orders are close to the theoretical  orders $1$ and $0.75$  obtained in  \thmref{mainresult1} for $\beta=2$ and $\beta=1.5$ respectively.

\end{document}